\documentclass[10pt,twoside]{amsart}
\usepackage{amsmath}
\usepackage{amssymb}
\usepackage{amsthm}
\usepackage{latexsym}
\usepackage{amscd}
\usepackage{xypic}
\xyoption{curve}
\usepackage{ifthen} 
\usepackage{hyperref} 
\usepackage{graphicx}

 \def\cocoa{{\hbox{\rm C\kern-.13em o\kern-.07em C\kern-.13em o\kern-.15em A}}}

\usepackage{xcolor}

\newtheorem{theorem}{Theorem}[section]
\newtheorem{lemma}[theorem]{Lemma}
\newtheorem{proposition}[theorem]{Proposition}

\newtheorem{remark}[theorem]{Remark}

\theoremstyle{definition}
\newtheorem{definition}[theorem]{Definition}

\newcommand {\Hom}{\mathrm{Hom}}
\newcommand {\Ext}{\mathrm{Ext}}
\newcommand {\im}{\mathrm{im}}
\newcommand {\rk}{\mathrm{rk}}

\newcommand {\Hilb}{\mathcal{H}\kern -0.25ex{\mathit ilb\/}}

\newcommand {\bN}{\mathbb{N}}
\newcommand {\bZ}{\mathbb{Z}}

\newcommand {\bP}{\mathbb{P}}

\newcommand{\cC}{{\mathcal C}}
\newcommand{\cE}{{\mathcal E}}

\newcommand{\cM}{{\mathcal M}}
\newcommand{\cN}{{\mathcal N}}
\newcommand{\cO}{{\mathcal O}}
\newcommand{\cI}{{\mathcal I}}

\newcommand{\Pic}{\operatorname{Pic}}

\def\p#1{{\bP^{#1}}}

\def\ga#1{{{\accent"12 #1}}}
 
\def\mapright#1{\mathbin{\smash{\mathop{\longrightarrow}
\limits^{#1}}}}

\title[rank two aCM bundles]{Rank two aCM bundles on the del Pezzo threefold with Picard number $3$}



\author[G. Casnati, D. Faenzi, F. Malaspina]{Gianfranco Casnati, Daniele Faenzi, Francesco Malaspina}
\thanks{All the authors are members of GRIFGA--GDRE project, supported by CNRS
  and INdAM, and of the GNSAGA group of INdAM. The first and third authors are supported by the framework of PRIN 2010/11 \lq Geometria delle variet\ga a algebriche\rq, cofinanced by MIUR. The second author is partially supported by ANR GEOLMI contract ANR-11-BS03-0011}

\begin{document}

\begin{abstract}
A del Pezzo threefold $F$ with maximal Picard number is isomorphic to $\p1\times\p1\times\p1$. In the present paper we completely classify locally free sheaves $\cE$ of rank $2$ 
such that $h^i\big(F,\cE(t)\big)=0$ for $i=1,2$ and $t\in\bZ$. Such a classification extends similar  results proved by E. Arrondo and L. Costa regarding del Pezzo threefolds with Picard number $1$.
\end{abstract}

\maketitle

\section{Introduction}
Let  $\p N$ be the $N$-dimensional projective space over an algebraically closed field $k$ of characteristic $0$. A
well--known theorem of Horrocks (see \cite{O--S--S} and the references
therein) states that a locally free sheaf $\cE$
on $\p
N$ splits as direct sum of invertible sheaves if and only if it has no
intermediate cohomology, i.e. $h^i\big(\p N,\cE(t)\big)=0$ for $0<i<N$
and $t\in\bZ$. 

It is thus natural to 
ask for the meaning of such a vanishing on other
kind of algebraic varieties $F$. Such a vanishing makes
sense only if a natural polarization is defined on $F$.
For instance, if there
is a natural embedding $F\subseteq\p N$, then one can consider
$\cO_F(h):=\cO_{\p N}(1)\otimes\cO_F$ and we can ask for locally free
sheaves $\cE$ on $F$ such that $  H^i_*\big(F,\cE\big):=\bigoplus_{t\in\bZ}H^i\big(F,\cE(t h)\big)=0$ for
$0<i<\dim(F)$ which are called arithmetically
Cohen-Macaulay (aCM for short) bundles. We are particularly interested in characterizing indecomposable aCM bundles, i.e. bundles of rank $r\ge2$ which do not split as sum of invertible sheaves.
Among aCM bundles, there are bundles $\cE$ such that $H^0_*\big(F,\cE\big)$ has the highest possible number of
generators in degree $0$. After \cite{Ul} such bundles are simply called Ulrich bundles. Ulrich bundles have many good properties, thus their description is of particular interest.

There are a lot of classical and recent papers devoted to the aforementioned
topics. For example, in the case of smooth quadrics and Grassmannians,
there are some classical result generalizing 
Horrocks' criterion by adding suitable vanishing to $\cE$ to force splitting
(see \cite{Kno}, \cite{A--O}, \cite{A--M}, and \cite{Ot}).
More recently, several authors considered cohomological splitting
criteria on hypersurfaces or Segre products, most notably for the first non-trivial
case, namely bundles of rank 2 (see
e.g. \cite{Ma1}, \cite{Ma2}, \cite{Ma3}, \cite{C--M1}, \cite{Ma4},
\cite{C--M2}, \cite{C--F}, \cite{MK--R--R1}, \cite{MK--R--R2}, \cite{B--M2}).
Ulrich bundles (even of higher rank) have been recently object of deep inspection (see, e.g. \cite{C--H1}, \cite{C--H2}, \cite{C--K--M}, \cite{C--M--P}).

Another case which deserves particular attention is that
of Fano and del Pezzo $n$--folds. We recall that a smooth
$n$--fold $F$ is Fano if its anticanonical sheaf $\omega_F^{-1}$ is
ample (see \cite{I--P} for results about Fano and del
Pezzo varieties). The greatest positive integer $r$ such that
$\omega_F\cong\cO_F(-r h)$ for some ample class $h\in\Pic(F)$ is called
the index of $F$ and one has
 $1\le r\le n+1$.
If $r=n-1$, then $F$ is
called del Pezzo. For such an $F$, the group $\Pic(F)$ is
torsion--free so that $\cO_F(h)$ is uniquely determined.
By definition, the degree of $F$ is the integer $d:=h^n$. In the case $n=3$ several results about aCM bundles on Fano and del Pezzo threefolds can be found in \cite{A--C}, \cite{A--G}, \cite{B--F1}, \cite{Ca}, \cite{Ma3}, \cite{Fa1}, \cite{A--F} for particular values of $d$. 

Let us restrict to the case of del Pezzo threefolds. It is well--known that $1\le
d\le8$. When $d\ge3$ the sheaf $\cO_F(h)$ is actually very ample so
$h$ is the hyperplane class of a natural embedding $F\subseteq\p {d+1}$.
These varieties are classified (see again \cite{I--P}).
For $3\le d\le 5$ and $d=8$, the complete classification of aCM
bundles of rank $2$ on $F$ can be found in \cite{A--C} and \cite{Ca} (where one can also find some generalization to the cases  $d=1,2$).

A fundamental hypothesis in almost all the aforementioned papers is that the varieties always have Picard number
$\varrho(F):=\rk(\Pic(F))=1$. Indeed in this case both $c_1$ and $c_2$
can be handled as integral numbers. 

The aim of the present paper is to give a complete classification of aCM bundle of rank
$2$ on $F:=\p1\times\p1\times\p1\subseteq \p7$. This case is at the opposite
extreme. Indeed we have $\varrho(F)=3$ and all the other del Pezzo threefolds have Picard number at most $2$.

To formulate our main result, we introduce a bit of
notation now. In what follows we will denote by $A(X)$  the Chow ring of a variety $X$, so that $A^r(X)$ denotes the set of  cycles of codimension $r$. 

We have three different projections $\pi_i\colon F\to\p1$ and we denote
by $h_i$ the pull--back in $A^1(F)$ of the class of a point in the
$i$--th copy of $\p1$. The exterior product morphism
$A(\p1)\otimes A(\p1)\otimes A(\p1)\to A(F)$ is an isomorphism (see
\cite{Fu}, Example 8.3.7), thus we finally obtain 
$$
A(F)\cong\bZ[h_1,h_2,h_3]/(h_1^2,h_2^2,h_3^2).
$$
In particular, if $\cE$ is any bundle of rank $2$ on $F$, then we can write $c_1(\cE)=\alpha_1h_1+\alpha_2h_2+\alpha_3h_3$ for suitable integers $\alpha_1,\alpha_2,\alpha_3$. There is an obvious action of the symmetric group of order $3$ on $F$ that lifts to an analogous action on $A(F)$. Thanks to such an action, we can often restrict our attention to the case $\alpha_1\le\alpha_2\le\alpha_3$.

Our first main result is the following.
\medbreak
\noindent
{\bf Theorem A.}
{\it Let $\cE$ be an indecomposable aCM bundle of rank $2$ on $F$ and let $c_1(\cE)=\alpha_1h_1+\alpha_2h_2+\alpha_3h_3$ with $\alpha_1\le\alpha_2\le\alpha_3$. 
Assume that $h^0\big(F,\cE\big)\ne0$ and $h^0\big(F,\cE(-h)\big)=0$ (we briefly say that $\cE$ is initialized).  Then:
  \begin{enumerate}
  \item the zero locus $E:=(s)_0$ of a general section $s\in H^0\big(F,\cE\big)$ has codimension $2$ inside $F$;
  \item either $0\le\alpha_i\le 2$, $i=1,2,3$, or $c_1=h_1+2h_2+3h_3$;
  \item if $c_1=2h$, or $c_1=h_1+2h_2+3h_3$, then $\cE$ is Ulrich.
  \end{enumerate}}
\medbreak

After a basic reminder on aCM bundles in Section \ref{saCM}, we give the proof
of this result, which occupies Sections
\ref{sLowerChern6-1}, \ref{sUpperChern6-1} and \ref{sTheorem}.

In order to better understand the above statement we recall that in \cite{A--C} the authors prove that for each indecomposable, initialized aCM bundle $\cE$ of rank $2$ on del Pezzo threefolds
of degree $d=3,4,5$, the bound $0\le c_1\le2$ holds.
Such a bound remains true also in the cases $d=1,2$.

Thus, from such a viewpoint, the statement of Theorem A can be
read as follows: either $c_1$
satisfies the  {\it standard bound} \lq\lq$0\le c_1\le2$\rq\rq, or $\cE$ is a {\it sporadic bundle}\/, i.e. $c_1=h_1+2h_2+3h_3$.

In Section \ref{sExtremal} we provide a complete classification of
bundles satisfying the standard bound. Then we examine the sporadic
case in Section \ref{sSporadic}. Finally, we will answer to the
natural question of determining which intermediate cases are actually
admissible in Section \ref{sIntermediate}. We summarize this second main result in the following simplified statement (see Theorems \ref{tLine}, \ref{tElliptic}, \ref{tRational}, \ref{tIntermediate} for an expanded and detailed statement).

\medbreak
\noindent
{\bf Theorem B.}
{\it There exists an indecomposable and initalized aCM bundle $\cE$ of rank $2$ on
  $F$ with
$c_1(\cE)=\alpha_1h_1+\alpha_2h_2+\alpha_3h_3$ with $\alpha_1\le\alpha_2\le\alpha_3$ if and only if  $(\alpha_1,\alpha_2,\alpha_3)$ is one of the following:
$$
(0,0,0),\quad (2,2,2),\quad (1,2,3),\quad (0,0,1),\quad(1,2,2).
$$
Moreover, denote by $E\subseteq F\subseteq\p7$ the zero--locus of a general section of such an $\cE$. Then:
  \begin{enumerate}
  \item if $\alpha_1=\alpha_2=\alpha_3=0$, then $E$ is a line and each such a curve on $F$ can be obtained in this way;
  \item if $\alpha_1=\alpha_2=\alpha_3=2$, then $E$ is a smooth elliptic normal curve of degree $8$ and each such a curve on $F$ can be obtained in this way;
  \item if $\alpha_i=i$, then $E$ is a rational normal curve of degree $7$ and each such a curve on $F$ can be obtained in this way;
 \item if $\alpha_1=\alpha_2=0$, $\alpha_3=1$, then $E$ is a line and each such a curve on $F$ can be obtained in this way;
 \item if $\alpha_1=1$, $\alpha_2=\alpha_3=2$, then $E$ is a, possibly reducible, quintic with arithmetic genus $0$.
  \end{enumerate}}
\medbreak

In particular, for each
line $E\subseteq F$  we show the existence of exactly four
indecomposable non--isomorphic aCM bundles of rank $2$ whose general section vanishes exactly
along $E$. 

The study of (semi)stability of aCM bundles of rank $2$ on $\p1\times\p1\times\p1$ and of the moduli spaces of such (semi)stable bundles are the object of the paper \cite{C--F--M1}.

The case of del Pezzo threefolds with Picard number $2$ will be the object of forthcoming papers.

Throughout the whole paper we refer to \cite{I--P} and \cite{Ha2} for all the unmentioned definitions, notations and results.

We conclude this introduction by expressing our thanks to the referee and to M. Filip for their comments that allowed us to considerably improve the exposition of the paper.

\section{Some facts on aCM locally free sheaves}
\label{saCM}
Throughout the whole paper $k$ will denote an algebraically closed field of characteristic $0$. Let $X\subseteq \p N$ be a subvariety, i.e. an integral closed subscheme defined over $k$. We set $\cO_X(h):=\cO_{\p N}(1)\otimes\cO_X$. We start this section by recalling two important definitions. 

The variety $X\subseteq \p N$ is called arithmetically Cohen--Macaulay (aCM for short) if and only if the natural restriction maps $H^0\big(\p N,\cO_{\p N}(t)\big)\to H^0\big(X,\cO_{X}(th)\big)$ are surjective and $h^i\big(X,\cO_{X}(th)\big)=0$, $1\le i\le \dim(X)-1$. 

The variety $X\subseteq \p N$ is called arithmetically Gorenstein (aG for short) if and only if it is aCM and $\alpha$--subcanonical, i.e. its dualizing sheaf satisfies $\omega_X\cong\cO_X(\alpha h)$ for some $\alpha\in \bZ$.

In what follows  $F$ will denote an aCM, integral and smooth subvariety of $\p N$ of positive dimension $n$. 

\begin{definition}
  Let $\cE$ be a vector bundle on $F$.
 We say that $\cE$ is {\sl arithmetically Cohen--Macaulay}
  (aCM for short) if  $  H^i_*\big(F,\cE\big)=0$ for each $i=1,\dots,n-1$.
\end{definition}

If $\cE$ is an aCM bundle, then the minimal number of
generators $m(\cE)$ of $H^0_*\big(F,\cE\big)$ as a module over the
graded coordinate ring of $F$ is $\rk(\cE)\deg(F)$ at
most (e.g. see \cite{C--H1}). The aCM bundles for which the maximum is
attained are worth of particular interest for several reasons: e.g. they are semistable and they form an open subset in their moduli space. 

Such bundles are called Ulrich
bundles in the sequel.
We recall the following definition (see \cite{C--H2}, Definition 2.1 and Lemma 2.2: see also \cite{C--H1}, Definition 3.4
and which is slightly different).

\begin{definition}
  Let $\cE$ be a vector bundle on $F$.
  We say that $\cE$ is {\sl initialized} if 
  $$
  \min\{\ t\in\bZ\ \vert\ h^0\big(F,\cE(t h)\big)\ne0\ \}=0.
  $$
  We say that $\cE$ is {\sl Ulrich} if it is initialized, aCM and
  $h^0\big(F,\cE\big)=\rk(\cE)\deg(F)$. 
\end{definition}

Let $\cE$ be an Ulrich bundle. On one hand we know that
$m(\cE)=\rk(\cE)\deg(F)$. On the other hand the generators of
$H^0\big(F,\cE\big)$ are minimal generators of $H^0_*\big(F,\cE\big)$
due to the vanishing of $H^0\big(F,\cE(-h)\big)$. We conclude that
$\cE$ is necessarily globally generated. Several other results are
known for Ulrich bundles (e.g. see \cite{E--S}, \cite{C--H1}, \cite{C--H2},
\cite{C--K--M}). 

Assume now $n=3$ and let $\cE$ be a bundle on $F$. We denote by
$\omega_i$ the Chern classes of the sheaf $\Omega_{F\vert k}^1$. In
this case Riemann--Roch theorem is 
\begin{equation}
  \label{RRgeneral}
  \begin{aligned}
    \chi(\cE)=-{\rk(\cE)\over24}\omega_1\omega_2&+{1\over6}(c_1(\cE)^3-3c_1(\cE)c_2(\cE)+3c_3(\cE))-\\
    &-{1\over4}(\omega_1c_1(\cE)^2-2\omega_1c_2(\cE))+{1\over{12}}(\omega_1^2c_1(\cE)+\omega_2c_1(\cE)).
  \end{aligned}
\end{equation}

Assume $\rk(\cE)=2$ and let $s$ be a global section of  $\cE$. In general its zero--locus
$(s)_0\subseteq F$ is either empty or its codimension is at most
$2$. Thus, in this second case, we can always write $(s)_0=E\cup D$
where $E$ has codimension $2$ (or it is empty) and $D$ has pure codimension
$1$ (or it is empty). In particular $\cE(-D)$ has a section vanishing
on $E$, thus we can consider its Koszul complex 
\begin{equation}
  \label{seqIdeal}
  0\longrightarrow \cO_F(D)\longrightarrow \cE\longrightarrow \cI_{E\vert F}(c_1-D)\longrightarrow 0.
\end{equation}
If $D=0$, then $E$ is locally complete intersection, being $\rk(\cE)=2$, hence it has no embedded components.

Moreover we also have the following exact sequence
\begin{equation}
  \label{seqStandard}
  0\longrightarrow \cI_{E\vert F}\longrightarrow \cO_F\longrightarrow \cO_E\longrightarrow 0.
\end{equation}

The above construction can be reversed. Indeed on a smooth threefold the following particular case of the more general Hartshorne--Serre correspondence holds (for further details about the statement in the general case see \cite{Vo}, \cite{Ha1}, \cite{Ar}).

\begin{theorem}
  \label{tSerre}
  Let $E\subseteq F$ be a local complete intersection subscheme of codimension $2$ and $\mathcal L$ an invertible sheaf on $F$ such that $H^2\big(F,{\mathcal L}^\vee\big)=0$. If $\det(\cN_{E\vert F})\cong\mathcal L\otimes\cO_E$, then there exists a vector bundle $\cE$ of rank $2$ on $F$ such that:
  \begin{enumerate}
  \item $\det(\cE)\cong\mathcal L$;
  \item $\cE$ has a section $s$ such that $E$ coincides with the zero locus $(s)_0$ of $s$.
  \end{enumerate}
  Moreover, if $H^1\big(F,{\mathcal L}^\vee\big)= 0$, the above two conditions  determine $\cE$ up to isomorphism.
\end{theorem}

From now on, $F$ will be $\p1\times\p1\times\p1$. 
We recall that the canonical sheaf satisfies $\omega
_F\cong\cO_F(-2h)$ where $h$ is the hyperplane class on $F$ given by
the natural embedding
$F\subseteq\p{7}$. Since $F$ is also aCM, it follows that it is also aG, thus a vector bundle $\cE$ on $F$ is aCM if and only if the same holds for ${\cE}^\vee$.

As pointed out in the introduction $A(F)\cong \bZ[h_1,h_2,h_3]/(h_1^2,h_2^2,h_3^2]$: thus $\Pic(F)$ is freely generated by $h_1$, $h_2$, $h_3$.  We have  $h=h_1+h_2+h_3$, $\deg(F)=h^3=6$,
$\omega_1=\omega_F=-2h$, $\omega_1\omega_2=-24$ in Formula
\eqref{RRgeneral}. More generally, for each $D\in \Pic(F)$, then there are $\delta_1,\delta_2,\delta_3\in\bZ$ such that $\cO_F(D)\cong\cO_F(\delta_1h_1+\delta_2h_2+\delta_3h_3)$. The following K\"unneth's formulas will be
repeatedly used without explicit mention throughout the paper 
\begin{equation*}
  \label{Kunneth}
    h^i\big(F,\cO_F(D)\big)=\sum_{(i_1,i_2,i_3)\in\bN^3,\atop{i_1+i_2+i_3=i}}h^{i_1}\big(F,\cO_{\p1}(\delta_1)\big)h^{i_2}\big(F,\cO_{\p1}(\delta_2)\big)h^{i_3}\big(F,\cO_{\p1}(\delta_3)\big).
\end{equation*}
We will write $D\ge0$ to denote that $D$ has sections, i.e. $\delta_j\ge0$, $j=1,2,3$.

In particular we immediately have the following result.

\begin{lemma}
  \label{lInvaCM}
  The initialized aCM bundles of rank $1$ on $F$ are:
  \begin{gather*}
    \cO_F,\quad \cO_F(h_1),\quad\cO_F(h_2),\quad\cO_F(h_3),\\
    \cO_F(h_1+h_2),\quad\cO_F(h_1+h_3),\quad\cO_F(h_1+h_2),\\
    \cO_F(2h_1+h_2),\quad\cO_F(2h_1+h_3),\quad\cO_F(h_1+2h_2),\\
    \cO_F(2h_2+h_3),\quad\cO_F(h_1+2h_3),\quad\cO_F(h_2+2h_3).
  \end{gather*}
\end{lemma}
\begin{proof}
  Straightforward.
\end{proof}

We now turn our attention on rank $2$ vector bundles $\cE$ on $F$. 
In what follows, we will denote its  Chern classes by
\begin{gather*}
  c_1:=c_1(\cE)=\alpha_1h_1+\alpha_2h_2+\alpha_3h_3,\\ 
  c_2:=c_2(\cE)=\beta_1h_2h_3+\beta_2h_1h_3+\beta_3h_1h_2.
\end{gather*}

\section{A lower bound on the first Chern class}\label{sLowerChern6-1}
In this section we will find a bound from below for the first Chern
class $c_1$ of an indecomposable initialized aCM bundle $\cE$ of rank
$2$ on $F$. 

The following lemma will be useful.
\begin{lemma}
  \label{lGG6-1}
  Let $\cE$ be an initialized aCM bundle of rank $2$ on $F$. Then $\cE^\vee(2h)$ is globally generated.
\end{lemma}
\begin{proof}
  We have $h^i\big(F,\cE^\vee((2-i)h)\big)=h^{3-i}\big(F,\cE((i-4)h)\big)=0$, $i=1,2,3$. It follows that $\cE$ is $2$--regular in the sense of Castelnuovo--Mumford (see \cite{Mu}), hence it follows the first assertion. 
\end{proof}

\begin{remark}
\label{rBoundAlpha}
An immediate consequence of Lemma \ref{lGG6-1} is that $4h-c_1=c_1(\cE^\vee(2h))$ is effective, hence $\alpha_i\le4$. 

If $\alpha_i\ge3$, $i=1,2,3$, then there would be an injective morphism $\cO_F(2h-c_1)\to\cO_F(-h)$, which would induce an injective morphism $H^0\big(F,\cE(2h-c_1)\big)\to H^0\big(F,\cE(-h)\big)$. On the one hand we know that the target space is zero. On the other hand $H^0\big(F,\cE(2h-c_1)\big)\ne0$ since $\cE(2h-c_1)\cong{\cE}^\vee(2h)$ is globally generated, a contradiction. It follows that at least one of the $\alpha_i$ is at most $2$. 
\end{remark}

We now concentrate our attention in the proof that each non--zero section of an indecomposable initialized aCM bundle of rank $2$ vanishes on $F$ exactly along a curve. We first check that its zero--locus is non--empty.

\begin{lemma}
  \label{lNonEmpty6-1}
  Let $\cE$ be an indecomposable initialized aCM bundle of rank $2$ on $F$. Then the zero locus $(s)_0\subseteq F$ of a section of $\cE$ is non--empty. 
\end{lemma}
\begin{proof}
  Assume that $\alpha_1\le\alpha_2\le\alpha_3$. If $(s)_0=\emptyset$, then sequence \eqref{seqIdeal} becomes
  $$
  0\longrightarrow \cO_F\longrightarrow \cE\longrightarrow \cO_F(c_1)\longrightarrow 0.
  $$
  Such a sequence corresponds to an element of $\Ext^1\big(\cO_F(c_1),\cO_F\big)=H^1\big(F,\cO_F(-c_1)\big)$. Since $\cE$ is indecomposable, it follows that the last space must be non--zero. Thus $\alpha_3\ge2$ and $\alpha_1\le\alpha_2\le0$.

  If either $\alpha_1\le-1$ or $\alpha_3\ge3$, then the cohomology of the above sequence twisted by $-2h$ gives
  \begin{align*}
    0\longrightarrow H^2\big(F,\cE(-2h)\big)\longrightarrow H^2\big(F,\cO_F(c_1-2h)\big)\longrightarrow H^3\big(F,\cO_F(-2h)\big)\cong k.
  \end{align*}
  Due to the hypothesis on $c_1$ we have $1\le h^2\big(F,\cE(-2h)
  \big)$, contradicting that $\cE$ is aCM. We conclude that $\alpha_1=\alpha_2=0$, i.e. $c_1=2h_3$. 

  Since $h^1\big(F,\cO_F(-2h_3)\big)=1$ there exists exactly one non--trivial exact sequence of the form
  $$
  0\longrightarrow \cO_F(-2h_3)\longrightarrow \cE(-2h_3)\longrightarrow \cO_F\longrightarrow 0.
  $$
  Notice that we always have on $F$ the pull--back via $\pi_3^*$ of the Euler exact sequence, i.e. 
  $$
  0\longrightarrow \cO_F(-2h_3)\longrightarrow \cO_F(-h_3)\oplus \cO_F(-h_3)\longrightarrow \cO_F\longrightarrow 0.
  $$
  We conclude that the two sequences above are isomorphic, hence
  $\cE\cong\cO_F(h_3)\oplus \cO_F(h_3)$. It follows that there are no
  indecomposable initialized aCM vector bundles of rank $2$ on $F$
  with $c_1=2h_3$.
 \end{proof}

Lemma \ref{lNonEmpty6-1} implies that for each $s\in H^0\big(F,\cE\big)$ we have $(s)_0=E\cup D$ where $E$ has codimension $2$ (or it is empty) and $D\in\vert\delta_1h_1+\delta_2h_2+\delta_3h_3\vert$ has codimension $1$ (or it is empty). 

On the one hand, let us take the cohomology of Sequence \eqref{seqIdeal} twisted by $\cO_F(-h)$. Then the vanishing of $h^0\big(F,\cE(-h)\big)$
implies $h^0\big(F,\cO_F(D-h)\big)=0$. In particular we know that at
least one of the $\delta_i$ is zero.

On the other hand, twisting the same sequence by $\cO_F(2h-c_1)$, using the isomorphism $\cE(-c_1)\cong{\cE}^\vee$ and taking into account that ${\cE}^\vee(2h)$ is globally generated (see Lemma \ref{lGG6-1}), we obtain that $\cI_{E\vert F}(2h-D)$ is globally generated too. Thus
$$
0\ne H^0\big(F, \cI_{E\vert F}(2h-D)\big)\subseteq H^0\big(F, \cO_F(2h-D)\big),
$$
hence $\delta_i\le2$, $i=1,2,3$. Thus, up to permutations, 
\begin{equation}
\label{delta}
(\delta_1,\delta_2,\delta_3)\in\left\{\ (0,0,0), (0,0,1), (0,1,1), (0,0,2), (0,1,2), (0,2,2)\ \right\}.
\end{equation}

We now prove that $c_1-D$ is effective: notice that this implies the effectiveness of $c_1$ too. Assume $c_1-D$ is non--effective on $F$: thanks to Sequence \eqref{seqStandard} we have
\begin{equation}
\label{L0}
h^0\big(F,\cI_{E\vert F}(c_1-D)\big)\le h^0\big(F,\cO_{F}(c_1-D)\big)=0.
\end{equation}
  Taking into account that $\cE$ is aCM and $\delta_i\ge0$, the cohomology of Sequence \eqref{seqIdeal} yields also the vanishings
\begin{equation}
\label{L1}
  h^1\big(F,\cI_{E\vert F}(c_1-D)\big)=h^2\big(F,\cI_{E\vert F}(c_1-D)\big)=0.
\end{equation}
  Twisting Sequence \eqref{seqIdeal} by $\cO_F(-h)$, the same argument also gives
\begin{equation}
\label{L2}
  h^0\big(F,\cI_{E\vert F}(c_1-D-h)\big)=h^1\big(F,\cI_{E\vert F}(c_1-D-h)\big)=\\
h^2\big(F,\cI_{E\vert F}(c_1-D-h)\big)=0.
  \end{equation}

Finally, the cohomology of Sequence \eqref{seqIdeal} twisted by $\cO_F(-2h)$ gives
$$
0\longrightarrow H^1\big(F,\cI_{E\vert F}(c_1-D-2h)\big)\longrightarrow H^2\big(F,\cO_F(D-2h)\big)\longrightarrow 0
$$
because $\cE$ is aCM. Let
$$
\eta_1(D):=\left\lbrace\begin{array}{ll} 
0\quad&\text{if $D\not\in\left\vert2h_j\right\vert$,}\\
1\quad&\text{if $D\in\left\vert2h_j\right\vert$.}
\end{array}\right.
$$
Due to the restrictions on the $\delta_i$'s, we immediately have
\begin{equation}
\label{L3}
h^1\big(F,\cI_{E\vert F}(c_1-D-2h)\big)=h^2\big(F,\cO_F(D-2h)\big)=\eta_1(D).
\end{equation}
We also have the exact sequence
\begin{align*}
0\longrightarrow H^2\big(F,\cI_{E\vert F}(c_1-D-2h)\big)&\longrightarrow H^3\big(F,\cO_F(D-2h)\big)\longrightarrow \\
\longrightarrow H^3\big(F,\cE(-2h)\big)&\longrightarrow H^3\big(F,\cI_{E\vert F}(c_1-D-2h)\big)\longrightarrow 0.
\end{align*}
If $D\ne0$, then $h^3\big(F,\cO_F(D-2h)\big)=0$, thus $h^2\big(F,\cI_{E\vert F}(c_1-D-2h)\big)=0$ too. Assume $D=0$: then $E\ne\emptyset$ (see Lemma \ref{lNonEmpty6-1}), so that $h^0\big(F,\cI_{E\vert F}\big)=0$, hence
\begin{align*}
h^3\big(F,\cE(-2h)\big)&=h^0\big(F,\cE(-c_1)\big)=h^0\big(F,\cO_F(-c_1)\big)=\\
&=h^3\big(F,\cO_F(c_1-2h)\big)=h^3\big(F,\cO_F(c_1-D-2h)\big).
\end{align*}
Sequence \eqref{seqStandard} twisted by $\cO_F(c_1-D-2h)$, also yields $h^3\big(F,\cI_{E\vert F}(c_1-D-2h)\big)=
h^3\big(F,\cO_F(c_1-D-2h)\big)$. Thus the last map in the above sequence is an isomorphism. We conclude that if we define
$$
\eta_2(D):=\left\lbrace\begin{array}{ll} 
0\quad&\text{if $D\ne0$,}\\
1\quad&\text{if $D=0$,}
\end{array}\right.
$$
then
\begin{equation}
\label{L4}
h^2\big(F,\cI_{E\vert F}(c_1-D-2h)\big)=h^3\big(F,\cO_F(D-2h)\big)=\eta_2(D).
\end{equation}

\begin{lemma}
  Let $\cE$ be an indecomposable initialized aCM bundle of rank $2$ on $F$. Assume that $s\in H^0\big(F,\cE\big)$ is such that $(s)_0=E\cup D$ where $E$ has codimension $2$ (or it is empty) and $D$ has codimension $1$. If $c_1-D$ is not effective, then $E=\emptyset$.
  \end{lemma}
  \begin{proof}
 Assume $E\ne\emptyset$, hence $\deg(E)\ge1$. Let $H$ be a general hyperplane in $\p7$. Define $S:=F\cap H$ and $Z:=E\cap H$, so that $\dim(Z)=0$. We have the following exact sequence
  \begin{equation}
    \label{seqSectionIdeal}
    0\longrightarrow \cI_{E\vert F}(c_1-D-h)\longrightarrow \cI_{E\vert F}(c_1-D)\longrightarrow \cI_{Z\vert S}(c_1-D)\longrightarrow 0.
  \end{equation}
  The vanishing of the cohomologies of $\cI_{E\vert F}(c_1-D)$ and $\cI_{E\vert F}(c_1-D-h)$ (see Equalities \eqref{L0}, \eqref{L1}, \eqref{L2}) implies
  $$
  h^0\big(S,\cI_{Z\vert S}(c_1-D)\big)=h^1\big(S,\cI_{Z\vert S}(c_1-D)\big)=0.
  $$
  The above equalities and the cohomology of
  \begin{equation}
    \label{seqSectionScheme}
  0\longrightarrow \cI_{Z\vert S}(c_1-D)\longrightarrow \cO_S(c_1-D)\longrightarrow \cO_Z\longrightarrow 0,
 \end{equation}
 give $h^1\big(S,\cO_{S}(c_1-D)\big)=0$ and
  $$
  h^0\big(S,\cO_{S}(c_1-D)\big)=h^0\big(Z,\cO_Z\big)=\deg(E)\ge1.
  $$

  Twisting Sequence \eqref{seqSectionIdeal} by $\cO_F(-h)$, thanks to \eqref{L3} and \eqref{L4}, a similar argument also gives
  \begin{gather*}
    h^0\big(S,\cI_{Z\vert S}(c_1-D-h)\big)=h^1\big(F,\cI_{E\vert F}(c_1-D-2h)\big)=\eta_1(D),\\
    h^1\big(S,\cI_{Z\vert S}(c_1-D-h)\big)=h^2\big(F,\cI_{E\vert F}(c_1-D-2h)\big)=\eta_2(D).
  \end{gather*}
  The above equalities and the cohomology of Sequence \eqref{seqSectionScheme} twisted by $\cO_S(-h)$ yield
  $$
h^0\big(Z,\cO_Z\big)= h^0\big(S,\cO_{S}(c_1-D-h)\big)-  h^1\big(S,\cO_{S}(c_1-D-h)\big)-\eta_1(D) +\eta_2(D),
  $$
and $h^1\big(S,\cO_{S}(c_1-D-h)\big)\le \eta_2(D)$. It follows 
  $$
 h^0\big(S,\cO_{S}(c_1-D-h)\big)-1\le  h^0\big(Z,\cO_Z\big)\le h^0\big(S,\cO_{S}(c_1-D-h)\big)+1.
 $$

Set $c_1-D:=\epsilon_1h_1+\epsilon_2h_2+\epsilon_3h_3$ and assume $\epsilon_1\le\epsilon_2\le\epsilon_3$: $\epsilon_1\le -1$ because $c_1-D$ is assumed to be not effective. Thus the cohomology of 
 $$
 0\longrightarrow\cO_F(c_1-D-h)\longrightarrow\cO_F(c_1-D)\longrightarrow\cO_S(c_1-D)\longrightarrow0
 $$
and the vanishing $h^1\big(S,\cO_{S}(c_1-D)\big)=0$ proved above, yield
$$
h^0\big(S,\cO_{S}(c_1-D)\big)=h^1\big(F,\cO_{F}(c_1-D-h)\big)-h^1\big(F,\cO_{F}(c_1-D)\big).
$$
Since $h^0\big(S,\cO_{S}(c_1-D)\big)\ge1$, it follows that $h^1\big(F,\cO_{F}(c_1-D-h)\big)\ge1$. In particular $1\le \epsilon_2\le\epsilon_3$, hence
$$
h^0\big(F,\cO_{F}(c_1-D)\big)=\epsilon_1\epsilon_2+\epsilon_1\epsilon_3+\epsilon_2\epsilon_3+\epsilon_1+\epsilon_2+\epsilon_3+1.
$$
Taking again the cohomology of the above sequence twisted by $\cO_F(c_1-D-h)$ we obtain the exact sequence (recall that $c_1-D$ is not effective)
\begin{align*}
 0\longrightarrow H^0(S,\cO_S(c_1-D-h)\big)&\longrightarrow H^1(F,\cO_F(c_1-D-2h)\big)\mapright\varphi\\
\longrightarrow& H^1(F,\cO_F(c_1-D-h)\big)\longrightarrow H^1(S,\cO_S(c_1-D-h)\big).
\end{align*}

If $\varphi$ is not surjective then necessarily $1\le h^1(S,\cO_S(c_1-D-h)\big)\le \eta_2(D)\le 1$, thus equality must hold. In particular $D=0$, $
h^0\big(S,\cO_{S}(c_1-D-h)\big)=h^0\big(Z,\cO_Z\big)\ge1$ and 
$$
h^0\big(S,\cO_{S}(c_1-D-h)\big)=h^1\big(F,\cO_{F}(c_1-D-2h)\big)-h^1\big(F,\cO_{F}(c_1-D-h)\big)+1.
$$
If $\epsilon_2=1$, then $h^1\big(F,\cO_{F}(c_1-D-2h)\big)=0$ and $h^1\big(F,\cO_{F}(c_1-D-h)\big)\ge1$, hence $h^0\big(S,\cO_{S}(c_1-D-h)\big)\le0$, a contradiction. Thus $2\le \epsilon_2\le\epsilon_3$, hence 
$$
h^0\big(S,\cO_{S}(c_1-D-h)\big)=\epsilon_1\epsilon_2+\epsilon_1\epsilon_3+\epsilon_2\epsilon_3-\epsilon_1-\epsilon_2-\epsilon_3+2.
$$
The equality $h^0\big(S,\cO_{S}(c_1-D)\big)=h^0\big(Z,\cO_Z\big)=h^0\big(S,\cO_{S}(c_1-D-h)\big)$ implies $2\epsilon_1+2\epsilon_2+2\epsilon_3+1=0$, which is not possible, because the $\epsilon_i$'s are integers.

If $\varphi$ is surjective, then 
$$
h^0\big(S,\cO_{S}(c_1-D-h)\big)=h^1\big(F,\cO_{F}(c_1-D-2h)\big)-h^1\big(F,\cO_{F}(c_1-D-h)\big),
$$
and we can easily again infer $2\le \epsilon_2\le\epsilon_3$, hence 
$$
h^0\big(S,\cO_{S}(c_1-D-h)\big)=\epsilon_1\epsilon_2+\epsilon_1\epsilon_3+\epsilon_2\epsilon_3-\epsilon_1-\epsilon_2-\epsilon_3+1.
$$
If $h^0\big(Z,\cO_Z\big)=h^0\big(S,\cO_{S}(c_1-D-h)\big)\pm1$ we obtain $2\epsilon_1+2\epsilon_2+2\epsilon_3\pm1=0$ which is not possible.

If $h^0\big(Z,\cO_Z\big)=h^0\big(S,\cO_{S}(c_1-D-h)\big)$, then $\epsilon_1+\epsilon_2+\epsilon_3=0$, thus $\epsilon_1=-\epsilon_2-\epsilon_3$. Substituting in the expression of $h^0\big(S,\cO_{S}(c_1-D)\big)$ we obtain 
  $$
  h^0\big(S,\cO_{S}(c_1-D)\big)=-\epsilon_2^2-\epsilon_2\epsilon_3-\epsilon_3^2+1\le -11
  $$
  an absurd.

We conclude that the assumption $E\ne\emptyset$ yields a contradiction. This completes the proof of the statement.
\end{proof}

We are now ready to prove the claimed effectiveness of $c_1-D$. 

\begin{proposition}
  \label{pLower6-1}
  If $\cE$ is an indecomposable initialized aCM bundle of rank $2$ on $F$, $s\in H^0\big(F,\cE\big)$ and $D$ denotes the component of codimension $1$ of $(s)_0$, if any, then $c_1\ge c_1-D\ge0$. 

  In particular if $c_1=0$, then the zero--locus of each section $s\in H^0\big(F,\cE\big)$ has codimension $2$. 
\end{proposition}
\begin{proof}

Assume again that $c_1-D$ is non--effective, so that  $E=\emptyset$ due to the above Lemma. Thus $\cI_{E\vert F}\cong\cO_F$, hence Sequence \eqref{seqIdeal} becomes
  \begin{equation}
    \label{seqIdealO}
    0\longrightarrow \cO_F(D)\longrightarrow \cE\longrightarrow \cO_{F}(c_1-D)\longrightarrow 0.
  \end{equation}
Thanks to Lemma \ref{lNonEmpty6-1} we know that $D$ is non--zero.   If $D\in\left\vert2h_2+2h_3\right\vert$, then the cohomology of Sequence \eqref{seqIdealO} twisted by $\cO_F(-2h)$ would give
  $$
  0=h^0\big(F,\cO_F(c_1-D)\big)\ge h^0\big(F,\cO_F(c_1-D-2h)\big)=h^1\big(F,\cO_F(D-2h)\big)=1,
  $$
  a contradiction. If $D\in\left\vert2h_3\right\vert$, then $h^1\big(F,\cO_F(c_1-D-2h)\big)=h^2\big(F,\cO_F(D-2h)\big)=1$ contradicting the condition $\epsilon_1\le-1$. 

  In the remaining cases we know that $h^i\big(F,\cO_F(D-h)\big)=h^i\big(F,\cO_F(D-2h)\big)=0$, $i=1,2$. Taking the cohomology of Sequence \eqref{seqIdealO} suitably twisted, we obtain $h^i\big(F,\cO_F(c_1-D-h)\big)=h^i\big(F,\cO_F(c_1-D-2h)\big)=0$ in the same range, because both $\cE$ and $\cO_F(D)$ are aCM. 

  If Sequence \eqref{seqIdealO} does not split, then
  $$
  H^1\big(F,\cO_F(2D-c_1)\big)=\Ext_{F}^1\big(\cO_{F}(c_1-D),\cO_F(D)\big)\ne0.
  $$
  Thus $\delta_j-\epsilon_j\le-2$ for exactly one $j=1,2,3$. Since
  $\delta_1\ge0$ and $\epsilon_1\le-1$ (as we assumed above) such a
  $j$ is not $1$. We can assume $\delta_3-\epsilon_3\le-2$, hence we
  obtain $\epsilon_3\ge2$. If $\epsilon_2\ge1$, then
  $h^1\big(F,\cO_F(c_1-D-h)\big)\ne0$. If $\epsilon_2\le0$, then
  $h^2\big(F,\cO_F(c_1-D-2h)\big)\ne0$. In both cases we have a contradiction. 
  \end{proof}

\section{An upper bound on the first Chern class}\label{sUpperChern6-1}
In this section we will find a bound from above for the first Chern class $c_1$ of an indecomposable initialized aCM bundle $\cE$ of rank $2$ on $F$. Thus $h^i\big(F,\cE^\vee(-h)\big)=h^{3-i}\big(F,\cE(-h)\big)=0$ for $i\ge1$. As pointed out in Lemma \ref{lNonEmpty6-1} the zero locus of a general section $s\in H^0\big(F,\cE\big)$ is non--empty. We will still decompose it as $(s)_0=E\cup D$, where $D$ is the divisorial part and $E$ has pure codimension $2$.

We have 
$$
h^0\big(F,\cI_{E\vert F}(-D-h)\big)\le h^0\big(F,\cO_{F}(-D-h)\big)=0
$$
(see Sequence \eqref{seqStandard}). Thus the cohomology of Sequence \eqref{seqIdeal} twisted by $\cO_F(-c_1)$, the standard isomorphism $\cE^\vee\cong\cE(-c_1)$ and  the effectiveness of $c_1-D$ proved in Proposition \ref{pLower6-1} imply
$$
\chi\big(\cE^\vee(-h)\big)=h^0\big(F,\cE^\vee(-h)\big)=h^0\big(F,\cO_F(D-c_1-h)\big)=0.
$$
Such an equality, Formula \eqref{RRgeneral} and the equalities
\begin{gather*}
  c_1^3=6\alpha_1\alpha_2\alpha_3,\\ 
  c_1^2h=2(\alpha_1\alpha_2+\alpha_1\alpha_3+\alpha_2\alpha_3),\\
  c_1h^2=2(\alpha_1+\alpha_2+\alpha_3),\\ 
  \omega_2c_1=4(\alpha_1+\alpha_2+\alpha_3),
\end{gather*}
finally imply
\begin{equation}
  \label{c_1c_2-6-1}
  c_1c_2=2\alpha_1\alpha_2\alpha_3.
\end{equation}
Since $\cE$ is initialized and aCM we also have $h^i\big(F,\cE(-2h)\big)=0$ for $i\le 2$. Thus $h^i\big(F,\cE^\vee\big)=h^{3-i}\big(F,\cE(-2h)\big)=0$, for $i\ge 1$. Let $D\ne0$: $H^0\big(F,\cI_{E\vert F}\big)$ is contained in the kernel of the natural injective map $H^0\big(F,\cO_F\big)\to H^0\big(E,\cO_E\big)$, hence $h^0\big(F,\cI_{E\vert F}\big)=h^0\big(F,\cI_{E\vert F}(-D)\big)=0$. If $D\ne0$, then
$$
h^0\big(F,\cI_{E\vert F}(-D)\big)\le h^0\big(F,\cO_{F}(-D)\big)=0.
$$
We conclude that $h^0\big(F,\cI_{E\vert F}(-D)\big)=0$ without restrictions on $D$. Let
$$
e(c_1,D):=\left\lbrace\begin{array}{ll} 
    0\quad&\text{if $D\ne c_1$,}\\
    1\quad&\text{if $D=c_1$.}
  \end{array}\right.
$$
Since $(s)_0\ne\emptyset$, it follows that 
$$
h^0\big(F,\cE^\vee\big)=h^0\big(F,\cO_F(D-c_1)\big)=e(c_1,D)
$$
Combining as above Sequences \eqref{seqStandard}, \eqref{seqIdeal} and Formulas \eqref{RRgeneral}, \eqref{c_1c_2-6-1} we finally obtain
\begin{equation}
  \label{hc_2-6-1}
  h c_2=\alpha_1\alpha_2\alpha_3+(1-\alpha_1)(1-\alpha_2)(1-\alpha_3)+1-e(c_1,D).
\end{equation}
In what follows we will often combine the above equalities with the trivial ones
\begin{equation}
  \label{AlphaBeta}
  \left\lbrace\begin{array}{l} 
      \alpha_1\beta_1+\alpha_2\beta_2+\alpha_3\beta_3=c_1c_2\\
      \beta_1+\beta_2+\beta_3=h c_2.
    \end{array}\right.
\end{equation}
Finally, assume that $E\ne\emptyset$. Its class in $A^2(F)$ is $c_2(\cE(-D))=c_2-c_1D+D^2$. Since $\vert h_i\vert$, $i=1,2,3$ is base--point--free on $F$, $D\ge0$ and $c_1-D\ge0$ (see Proposition \ref{pLower6-1}), it follows that
\begin{equation}
  \label{positivity}
  \beta_i\ge \beta_i-h_i(c_1-D)D=h_i c_2-h_i(c_1-D)D=h_i c_2(\cE(-D))\ge0,\qquad i=1,2,3.
\end{equation}
The numbers $c_1c_2$ and $h c_2$ will be very important in what follows.

\begin{lemma}
  \label{lInvariantE}
  Let $\cE$ be a vector bundle of rank $2$ on $F$ and $s\in H^0\big(F,\cE\big)$ a section such that $E:=(s)_0$ has codimension $2$. Then it has no embedded components and
  $$
  \deg(E)=h c_2,\qquad
  p_a(E)=\frac12c_1c_2-hc_2+1.
  $$
 \end{lemma}
\begin{proof}
We already noticed that if $E$ has codimension $2$, then it has no embedded components. 

The assertion on $\deg(E)$ is immediate, because the class in $A^2(F)$ is $c_2$. 
Let us compute $p_a(E)$. Adjunction formula on $F$ yields $\omega_E\cong\omega_F\otimes\det({\cN}_{E\vert F})$. We have ${\cN}_{E\vert F}\cong\cE\otimes\cO_E$, thus $\det({\cN}_{E\vert F})\cong\cO_E(c_1)$. We also have that $\omega_F\cong\cO_F(-2h)$. Equating the degrees of the two members of the above formula we obtain also the second equality.
\end{proof}

Now we turn our attention on the upper bound on $c_1$.

\begin{proposition}
  \label{pUpper6-1}
  If $\cE$ is an indecomposable initialized aCM bundle of rank $2$ on $F$, then either $2h-c_1\ge0$ or $c_1=h_1+2h_2+3h_3$ up to permutations of the $h_i$'s.
  
  Moreover, if either $c_1=2h$ or $c_1=h_1+2h_2+3h_3$ up to permutations of the $h_i$'s, then $\cE$ is Ulrich, hence the zero--locus $E:=(s)_0$ of a general section $s\in H^0\big(F,\cE\big)$ is a smooth curve. 
\end{proposition}
\begin{proof}
We distinguish two cases according to whether $\cE$ is regular in the sense of Castelnuovo--Mumford or not.

In the latter case we infer $h^0\big(F,\cE^\vee(h)\big)=h^3\big(F,\cE(-3h)\big)\ne0$ because $\cE$ is aCM. Thus, if $t\in\bZ$ is such that $\cE^\vee(th)$ is initialized, we know that $t\le1$. Since $\cE^\vee(th)$ is obviously also aCM, we know that  $2th-c_1=c_1(\cE^\vee(th))$ is effective, due to Proposition \ref{pLower6-1}. We conclude that the same is true for $2h-c_1$, because $t\le1$.

Now we assume that $\cE$ is regular, hence it is globally generated. Thus the zero--locus $E:=(s)_0$ of a general section $s\in H^0\big(F,\cE\big)$ is a smooth curve, i.e. the divisorial part $D$ is $0$. Let $\alpha_1\le\alpha_2\le\alpha_3$: we already know from Remark \ref{rBoundAlpha} that $0\le \alpha_1\le2$ and $\alpha_2,\alpha_3\le 4$. We want to examine the cases $\alpha_3=3,4$. In any case we have that $h^i\big(F,\cE^\vee(h)\big)=h^{3-i}\big(F,\cE(-3h)\big)=0$, $i=0,1,2,3$, because $\cE$ is assumed aCM, regular and initialized. Combining the above remarks with equality \eqref{RRgeneral}, Formulas  \eqref{c_1c_2-6-1}, \eqref{hc_2-6-1} and the vanishing of $e(c_1,D)$ (recall that $D=0$ and $c_1\ne0$), we finally obtain 
  $$
  0=\chi(\cE^\vee(h))=12-2\alpha_1-2\alpha_2-2\alpha_3.
  $$
  Thus we have the following possibilities for $(\alpha_1,\alpha_2,\alpha_3)$ with $\alpha_3\ge3$:  $(1,2,3)$, $(1,1,4)$, $(0,3,3)$, $(0,2,4)$.

  If $c_1=h_1+2h_2+3h_3$ up to permutations, then Formula
  \eqref{hc_2-6-1} give $h c_2=7$ and $c_1c_2=12$. Substituting in
  Equality \eqref{RRgeneral} we obtain $\chi(\cE)=12$. Since $\cE$ is
  aCM, it follows that $h^1\big(F,\cE\big)=h^2\big(F,\cE\big)=0$. Since $\cE$ is initialized, it follows that $h^3\big(F,\cE\big)=h^0\big(F,\cE(-2h-c_1)\big)=0$.  We
  conclude that $h^0\big(F,\cE\big)=12$ and
  $\cE$ turns out to be Ulrich. Similarly, one can also prove that $\cE$ is Ulrich in all the other cases.

We now prove that if  $c_1=\alpha_1h_1+\alpha_2h_2+\alpha_3h_3$ is either $h_1+h_2+4h_3$ or $2h_2+4h_3$ or $3h_2+3h_3$ up to permutations, then $\cE$ splits as a sum of invertible sheaves. 
  
  Let $c_1=h_1+h_2+4h_3$.  Then we know that ${\cE}^\vee(2h)$ is Ulrich (see Lemma 2.4 of \cite{C--H2}): in particular it is aCM and initialized and we have $c_1(\cE^\vee(2h))=3h_1+3h_2$. 
  
  Hence it suffices to examine only the two last cases $(0,3,3)$, $(0,2,4)$ where $\alpha_1=0$. Recall that $E$ is a smooth curve. Thus the System \eqref{AlphaBeta} and the Inequalities \eqref{positivity}  with $D=0$ force $c_2=ch_2h_3$ where $c:=\alpha_2\alpha_3-4$: due to the restrictions on $\alpha_2$ and $\alpha_3$ we have that $c$ is either $4$, or $5$. 
  
  Lemma \eqref{lInvariantE} implies $\deg(E)=c$ and $p_a(E)=1-c\le -3$. Since $E$ is smooth, it follows that it is necessarily the non--connected  union of $c$ pairwise skew lines whose class in $A^2(F)$ is $h_2h_3$.  In particular
  \begin{equation}
    \label{Decomp}
    \cO_E(a_1h_1+a_2h_2+a_3h_3)\cong\cO_{\p1}(a_1)^{\oplus c}.
  \end{equation}
  
Recalling the possible values of $c_1$ we have $h^0\big(F,\cO_F(c_1-h_2-2h_3)\big)=6$. On the one hand, the cohomology of Sequence \eqref{seqStandard} twisted by $\cO_F(c_1-h_2-2h_3)$ gives rise to a monomorphism 
$$
\vartheta\colon H^0\big(F,\cI_{E\vert F}(c_1-h_2-2h_3)\big) \to H^0\big(F,\cO_F(c_1-h_2-2h_3)\big). 
$$
We fix three points $p_1$, $p_2$, $p_3$ on three different lines in $E$. We have that $c_1-h_2-2h_3$ is either $h_2+2h_3$ or $2h_2+h_3$. 

Let $V_i\subseteq H^0\big(F,\cO_F(c_1-h_2-2h_3)\big)$ be the subspace of sections corresponding to divisors though the points $p_1,\dots,p_i$ . We have $V:=V_3\subseteq V_2\subseteq V_1\subseteq H^0\big(F,\cO_F(c_1-h_2-2h_3)\big)$. On the one hand, at each step the dimension can go down at most by $1$, hence $\dim(V)\ge h^0\big(F,\cO_F(c_1-h_2-2h_3)\big)-3=3$. On the other hand, using the fact that $\cO_F(h_i)$ is globally generated, it is easy to check that the above inclusions are all strict, thus we also have $\dim(V)\le3$. We conclude that $\dim(V)=3$.

We trivially have that $H^0\big(F,\cI_{E\vert F}(c_1-h_2-2h_3)\big)\cong\im(\vartheta)\subseteq V$, thus  $h^0\big(F,\cI_{E\vert F}(c_1-h_2-2h_3)\big)\le3$. Since $h^0\big(E,\cO_E(c_1-h_2-2h_3)\big)=c$ due to the Isomorphism \eqref{Decomp} above, it follows that $h^0\big(F,\cI_{E\vert F}(c_1-h_2-2h_3)\big)\ge6-c$.

The cohomology of Sequence \eqref{seqIdeal} for $s$ twisted by $\cO_F(-h_2-2h_3)$ thus yields
$$
3\ge h^0\big(F,\cE(-h_2-2h_3)\big)=h^0\big(F,\cI_{E\vert F}(c_1-h_2-2h_3)\big)\ge6-c\ge1.
$$
We have that $c_1(\cE(-h_2-2h_3))=c_1-2h_2-4h_3$ and $c_2(\cE(-h_2-2h_3))=0$. 

Let $\sigma\in H^0\big(F,\cE(-h_2-2h_3)\big)$ be a general section. As usual we can write $(\sigma)_0=\Theta\cup \Xi$ where $\Theta$ has codimension 2 (or it is empty) and $\Xi$ is a, possibly $0$, effective divisor. 

The cohomology of Sequence \eqref{seqStandard} yields
$$
h^0\big(F,\cI_{\Theta\vert F}(c_1-2h_2-4h_3-\Xi)\big)\le h^0\big(F,\cO_F(c_1-2h_2-4h_3-\Xi)\big)
$$
which is zero because $c_1$ is either $2h_2+4h_3$ or $3h_2+3h_3$ and $\Xi$ is effective. The cohomology of Sequence \eqref{seqIdeal} for $\sigma$ implies $h^0\big(F,\cO_F(\Xi)\big)=h^0\big(F,\cE(-h_2-2h_3)\big)\le3$. We deduce that $\Xi\in \vert\xi h_i\vert$ for some integer $0\le \xi\le2$. Moreover $\cE$ is initialized, thus we immediately infer from the monomorphism $\cO_F(\Xi)\to \cE(-h_2-2h_3)$ that $\Xi\in \vert\xi h_i\vert$ where $i\ne1$. 
It follows that
$$
c_2(\cE(-h_2-2h_3-\Xi))=c_2-c_1(h_2+2h_3+\Xi)+(h_2+2h_3+\Xi)^2={\left\lbrace\begin{array}{ll} 
    0\ &\text{if $\alpha_2=2$, $\alpha_3=4$,}\\
    -\xi h_2h_3\ &\text{if $\alpha_2=\alpha_3=3$, $\Xi=\xi h_3$,}\\
    \xi h_2h_3\ &\text{if $\alpha_2=\alpha_3=3$, $\Xi=\xi h_2$.}
  \end{array}\right.}
$$
Recall that $\deg(\Theta)=c_2(\cE(-h_2-2h_3-\Xi))h$. In the first case, $\deg(\Theta)=0$, hence $\Theta=\emptyset$. In the second case, $\deg(\Theta)=-\xi\le0$, hence $\xi=0$: in particular $\Xi=0$ and $\Theta=\emptyset$. In the third case Sequence \eqref{seqIdeal} for $\sigma$ becomes
$$
0\longrightarrow\cO_F(\xi h_2)\longrightarrow\cE(-h_2-2h_3)\longrightarrow\cI_{\Theta\vert F}((1-\xi)h_2-h_3)\longrightarrow0.
$$
Taking the cohomology of  the above sequence twisted by $\cO_F(-2h_1-h_2)$, we would obtain that $h^1\big(F,\cE(-2h)\big)\ge h^1\big(F,\cO_F(-2h_1+(\xi-1)h_2)\big)$, because $h^0\big(F,\cI_{\Theta\vert F}(-2h_1-\xi h_2-h_3)\big)=0$. Thus $\xi=0$ because $\cE$ is aCM. Since $\deg(\Theta)=0$, it again follows that $\Theta=\emptyset$. Thus, in any case
Sequence \eqref{seqIdeal} for $\sigma$ becomes
$$
0\longrightarrow\cO_F(\Xi)\longrightarrow\cE(-h_2-2h_3)\longrightarrow\cO_F(c_1-2h_2-4h_3-\Xi)\longrightarrow0.
$$
The above sequence corresponds to an element of 
$$
\Ext_F^1\big(\cO_F(c_1-2h_2-4h_3-\Xi),\cO_F(\Xi)\big)\cong H^1\big(F,\cO_F(2h_2+4h_3-c_1+2\Xi)\big).
$$
The last space is zero, due to the possible values of $c_1$ and the corresponding values of $\xi$. We conclude that $\cE$ would be decomposable, a contradiction. 
\end{proof}

\section{The proof of the theorem A}
\label{sTheorem}
In this section we will complete the proof of Theorem A stated in the introduction. It only remains to show that if $\cE$ is an indecomposable, initialized, aCM bundle, then its general section vanishes exactly along a curve (see Proposition \ref{pLower6-1}). 

We already checked in the previous sections (see Propositions \ref{pLower6-1} and  \ref{pUpper6-1}) that  either $c_1\ge0$ and $2h-c_1\ge0$ or $c_1=h_1+2h_2+3h_3$ up to permutations of the $h_i$'s for such kind of bundles and that the general section vanishes exactly along a curve when either $c_1=0$, or $2h-c_1=0$, or $c_1=h_1+2h_2+3h_3$ up to permutations of the $h_i$'s.

It follows that we can restrict our attention to the remaining cases. If we assume  $\alpha_1\le\alpha_2\le\alpha_3$, such cases satisfy $\alpha_1\le1\le\alpha_3$. 

\begin{lemma}
  \label{lggDivisor}
  Let $\cE$ be an indecomposable initialized aCM bundle of rank $2$ on $F$ whose general section $s\in H^0\big(F,\cE\big)$ satisfies $(s)_0=E\cup D$ where $E$ has codimension $2$ (or it is empty) and $D\in\vert\delta_1h_1+\delta_2h_2+\delta_3h_3\vert$ is non--zero. Then:
  \begin{enumerate}
  \item  up to permutations
  $$
  (\delta_1,\delta_2,\delta_3)\in\left\{\ (0,0,1), (0,1,1), (0,0,2), (0,1,2)\ \right\};
  $$
  \item $\cE$ is not globally generated;
  \item $E\ne\emptyset$.
\end{enumerate}
\end{lemma}
\begin{proof}
We already know that 
$$
(\delta_1,\delta_2,\delta_3)\in\left\{\ (0,0,0), (0,0,1), (0,1,1), (0,0,2), (0,1,2), (0,2,2)\ \right\}
$$
(see Relation \eqref{delta}). Since $D\ne0$, it follows that $(\delta_1,\delta_2,\delta_3)\ne(0,0,0)$. Let $D\in \left\vert2h_2+2h_3\right\vert$ and look at the cohomologies of Sequences \eqref{seqIdeal} and \eqref{seqStandard} respectively twisted by $\cO_F(-2h)$ and $\cO_F(c_1-D-2h)$. On the one hand, taking into account that $h^0\big(F,\cE(-2h)\big)=h^1\big(F,\cE(-2h)\big)=0$ by hypothesis we obtain 
$$
1=h^1\big(F,\cO_F(D-2h)\big)=h^0\big(F,\cI_{E\vert F}(c_1-D-2h)\big)\le h^0\big(F,\cO_F(c_1-D-2h)\big).
$$
On the other hand we also know that the last dimension is zero, due to the restrictions $D\ge0$, $2h-c_1\ge0$ and $c_1\ne2h$. It follows that equality $(\delta_1,\delta_2,\delta_3)=(0,2,2)$ cannot occur.

  If $\cE$ would be globally generated, then the general section $s\in H^0\big(F,\cE\big)$ should vanish on a curve, contradicting the hypothesis. 
  
  Finally, assume that $E=\emptyset$. Thus $\cI_{E\vert F}(c_1-D)\cong\cO_F(c_1-D)$ in Sequence  \eqref{seqIdeal}. We obviously have that $\cO_F(D)$ is globally generated and $h^1\big(F,\cO_F(D)\big)=0$. Moreover $\cO_F(c_1-D)$ is globally generated too by Proposition \ref{pLower6-1}. We conclude that $\cE$ is globally generated too thanks to Sequence  \eqref{seqIdeal} contradicting what we just proved above.
\end{proof}

From now on in this section we will assume that $\cE$ satisfies the hypothesis of the above lemma, hence we can assume $E\ne\emptyset$. We want to classify all the possible cases. To this purpose we recall that $E$ is in the class $c_2(\cE(-D))=c_2-c_1D+D^2\in A^2(F)$. Once the values of $(\alpha_1,\alpha_2,\alpha_3)$ and $(\delta_1,\delta_2,\delta_3)$ are fixed, the class $c_2$ can be computed thanks to System \eqref{AlphaBeta} and Formulas \eqref{c_1c_2-6-1}, \eqref{hc_2-6-1}. Moreover we have
\begin{gather*}
  c_1D=(\delta_2\alpha_3+\delta_3\alpha_2)h_2h_3+(\delta_1\alpha_3+\delta_3\alpha_1)h_1h_3+(\delta_1\alpha_2+\delta_2\alpha_1)h_1h_2,\\
  D^2=2(\delta_2\delta_3h_2h_3+\delta_1\delta_3h_1h_3+\delta_1\delta_2h_1h_2).
\end{gather*}
If $D\ne0$, then Inequalities \eqref{positivity} give further restrictions on the $\beta_i$'s. 

For example let $(\alpha_1,\alpha_2,\alpha_3)=(1,2,2)$ and $(\delta_1,\delta_2,\delta_3)=(0,1,2)$. In this case $e(c_1,D)=0$, hence System \eqref{AlphaBeta} is 
$$
\left\lbrace\begin{array}{l} 
\beta_1+2\beta_2+2\beta_3=8\\
\beta_1+\beta_2+\beta_3=5.
\end{array}\right.
$$
Inequalities \eqref{positivity} imply $\beta_i\ge0$, $i=1,2,3$, thus the only possible solutions $(\beta_1,\beta_2,\beta_3)$ of the above system are:
$$
(2,1,2),\quad(2,2,1),\quad(2,0,3),\quad(2,3,0).
$$
Using the expressions above for $c_1D$ and $D^2$, an immediate computation shows that the corresponding values of $c_2(\cE(-D))$ are $-h_1h_3+h_1h_2$, $0$, $-2h_1h_3+2h_1h_2$, $h_1h_3-h_1h_2$ respectively. Taking into account that $c_2(\cE(-D))$ is the class in $A^2(F)$ of  a non--empty scheme of pure codimension $2$, namely $E$, it follows that none of the above cases are admissible.

By examining along the same lines the all the possible cases, we obtain $53$ cases up to permutations of the $h_i$'s. We list all these cases in the following table. 

$$
\begin{array}{|c|c|c|c|c|}           \hline
  { (\alpha_1,\alpha_2,\alpha_3) } & {(\delta_1,\delta_2,\delta_3) } & e(c_1,D) & {(\beta_1,\beta_2,\beta_3) } & \text{class of $E$ in $A^2(F)$}\\ \hline
(0,0,1) & (0,0,1) & 1 &(0,0,0) & 0 \\  \hline
(0,0,2) & (0,0,1) & 0 & (0,0,0) & 0 \\  \hline
(0,1,1) & (0,0,1) & 0 & (1,0,0) & 0 \\  \hline
(0,1,1) & (0,1,0) & 0 & (1,0,0) & 0 \\  \hline
(0,1,1) & (0,1,1) & 1 & (0,0,0) & 0 \\  \hline
(0,1,2) & (0,0,1) & 0 & (1,0,0) & 0 \\  \hline
(0,1,2) & (0,1,0) & 0 & (1,0,0) & -h_2h_3 \\  \hline
(0,1,2) & (0,1,1) & 0 & (1,0,0) & 0 \\  \hline
(0,1,2) & (0,1,2) & 1 & (0,0,0) & 0 \\  \hline
(0,2,2) & (0,0,1) & 0 & (2,0,0) & 0 \\  \hline
(0,2,2) & (0,0,2) & 0 & (2,0,0) & -2h_2h_3 \\  \hline
(0,2,2) & (0,1,1) & 0 & (2,0,0) & 0 \\  \hline
(0,2,2) & (0,1,2) & 0 & (2,0,0) & 0 \\  \hline
(1,1,1) & (0,0,1) & 0 & (0,1,1) & -h_2h_3+h_1h_2 \\  \hline
(1,1,1) & (0,0,1) & 0 & (1,1,0) & 0 \\  \hline
(1,1,1) & (0,1,1) & 0 & (0,1,1) & 0 \\  \hline
(1,1,1) & (0,1,1) & 0 & (1,0,1) & h_2h_3-h_1h_3 \\  \hline
(1,1,2) & (0,0,1) & 0 & (1,1,1) & h_1h_3 \\  \hline
(1,1,2) & (0,0,1) & 0 & (0,2,1) & -h_2h_3+h_1h_3 \\  \hline
(1,1,2) & (0,1,0) & 0 & (0,1,1) & -h_2h_3+h_1h_3 \\  \hline
(1,1,2) & (0,1,0) & 0 & (0,2,1) & 0 \\  \hline
(1,1,2) & (0,1,1) & 0 & (1,1,1) & 0 \\  \hline
(1,1,2) & (0,1,1) & 0 & (0,2,1) & -h_2h_3+h_1h_3 \\  \hline
(1,1,2) & (1,0,1) & 0 & (0,2,1) & -h_2h_3+h_1h_3 \\  \hline
(1,1,2) & (1,1,0) & 0 & (0,2,1) & -2h_2h_3+h_1h_2 \\  \hline
(1,1,2) & (0,1,2) & 0 & (1,1,1) & h_2h_3-h_1h_3 \\  \hline
(1,1,2) & (0,1,2) & 0 & (0,2,1) & 0 \\  \hline
\end{array}
$$
$$
\begin{array}{|c|c|c|c|c|}           \hline
  { (\alpha_1,\alpha_2,\alpha_3) } & {(\delta_1,\delta_2,\delta_3) } & e(c_1,D) & {(\beta_1,\beta_2,\beta_3) } & \text{class of $E$ in $A^2(F)$}\\ \hline
(1,2,2) & (0,0,1) & 0 & (2,1,2) & 2h_1h_3 \\  \hline
(1,2,2) & (0,0,1) & 0 & (2,2,1) & h_1h_3+h_1h_2  \\  \hline
(1,2,2) & (0,0,1) & 0 & (2,0,3) & -h_1h_3+3h_1h_2 \\  \hline
(1,2,2) & (0,0,1) & 0 & (2,3,0) & 2h_1h_3 \\  \hline
(1,2,2) & (1,0,0) & 0 & (2,1,2) & 2h_2h_3-h_1h_3 \\  \hline
(1,2,2) & (1,0,0) & 0 & (2,2,1) & 2h_2h_3-h_1h_2 \\  \hline
(1,2,2) & (1,0,0) & 0 & (2,0,3) & 2h_2h_3-2h_1h_3+h_1h_2 \\  \hline
(1,2,2) & (0,0,2) & 0 & (2,1,2) & -2h_2h_3-h_1h_3+2h_1h_2 \\  \hline
(1,2,2) & (0,0,2) & 0 & (2,2,1) & -2h_2h_3+h_1h_2 \\  \hline
(1,2,2) & (0,0,2) & 0 & (2,0,3) & -2h_2h_3-2h_1h_3+3h_1h_2 \\  \hline
(1,2,2) & (0,0,2) & 0 & (2,3,0) & -2h_2h_3+h_1h_3 \\  \hline
(1,2,2) & (0,1,1) & 0 & (2,2,1) & h_1h_2 \\  \hline
(1,2,2) & (0,1,1) & 0 & (2,3,0) & 2h_1h_3-h_1h_2 \\  \hline
(1,2,2) & (1,0,1) & 0 & (2,1,2) & 0 \\  \hline
(1,2,2) & (1,0,1) & 0 & (2,2,1) & h_1h_3-h_1h_2 \\  \hline
(1,2,2) & (1,0,1) & 0 & (2,3,0) & 2h_1h_3-2h_1h_2 \\  \hline
(1,2,2) & (1,0,1) & 0 & (2,0,3) & -h_1h_3+h_1h_2 \\  \hline
(1,2,2) & (0,1,2) & 0 & (2,1,2) & -h_1h_3+h_1h_2 \\  \hline
(1,2,2) & (0,1,2) & 0 & (2,2,1) & 0 \\  \hline
(1,2,2) & (0,1,2) & 0 & (2,3,0) & h_1h_3-h_1h_2 \\  \hline
(1,2,2) & (0,1,2) & 0 & (2,0,3) & -2h_1h_3+2h_1h_2 \\  \hline
(1,2,2) & (1,0,2) & 0 & (2,1,2) & -2h_2h_3+h_1h_3 \\  \hline
(1,2,2) & (1,0,2) & 0 & (2,2,1) & -2h_2h_3+2h_1h_3-h_1h_2 \\  \hline
(1,2,2) & (1,0,2) & 0 & (2,0,3) & -2h_2h_3+3h_1h_3-2h_1h_2 \\  \hline
(1,2,2) & (1,0,2) & 0 & (2,0,3) & -2h_2h_3+h_1h_2 \\  \hline
\end{array}
$$
\vskip0.5truecm

Only few cases are actually admissible in the above table and we will examine them in the following.

When the class of $E$ is $h_1h_2$, then $E$ is a line and we have an exact sequence of the form
$$
0\longrightarrow\cO_F(-h_1-h_2)\longrightarrow \cO_F(-h_1)\oplus \cO_F(-h_2)\longrightarrow\cI_{E\vert F}\longrightarrow0.
$$
It follows that $\cI_{E\vert F}(h)\cong\cI_{E\vert F}(c_1-D)$ is globally generated. Thus we conclude again that $\cE$ should be globally generated, contradicting Lemma \ref{lggDivisor}. A similar argument holds when the class of $E$ is $h_1h_3$.

When the class of $E$ is $h_1h_3+h_1h_2$, then we cannot repeat the above argument. In this case we have that $h^2\big(F,\cO_F(D-h)\big)=h^3\big(F,\cO_F(D-2h)\big)=0$. Moreover $\cE$ is aCM, hence we obtain 
$$
h^1\big(F,\cI_{E\vert F}(c_1-D-h)\big)=h^2\big(F,\cI_{E\vert F}(c_1-D-2h)\big)=0.
$$
The cohomology of Sequence \eqref{seqStandard} twisted by $\cO_F(c_1-D-3h)$ implies that
$$
h^3\big(F,\cI_{E\vert F}(c_1-D-3h)\big)\le h^3\big(F,\cO_{F}(c_1-D-3h)\big).
$$
Since $c_1-D-3h=-2h_1-h_2-2h_3$ we conclude that the last dimension is zero. It follows that $\cI_{E\vert F}(c_1-D)$ is regular in the sense of Castelnuovo--Mumford, hence it is globally generated. As in the previous cases we also obtain that $\cE$ is globally generated, contradicting again Lemma \ref{lggDivisor}. A similar argument also proves that the class of $E$ cannot be $2h_1h_2$, $2h_1h_3$.

We summarize the above discussion in the following statement.

\begin{proposition}
\label{pCodimension}
  If $\cE$ is an indecomposable initialized aCM bundle of rank $2$ on $F$, then the zero--locus $(s)_0$ of a general section $s\in H^0\big(F,\cE\big)$ has pure dimension $2$. 
\end{proposition}

By combining Propositions \ref{pLower6-1}, \ref{pUpper6-1} and \ref{pCodimension} we have the proof of Theorem A.

The next sections are devoted to the classification of all indecomposable initialized aCM bundles of rank $2$ on $F$, or, in other words, to the proof of Theorem B stated in the introduction. 

\section{The standard bound}
\label{sExtremal}
In this section we deal with the extremal cases $c_1=0$ and $c_1=2h$. 
 In the first case Proposition \ref{pLower6-1} implies that the zero--locus  $E:=(s)_0$ of a general section $s\in H^0\big(F,\cE\big)$ has codimension $2$ inside $F$. Formula \eqref{hc_2-6-1} and Lemma \ref{lInvariantE} (notice that we know that $D=0$) give $\deg(E)=h c_2=1$, thus $E$ is a line. In particular, if its  class in $A^2(F)$ is $\beta_1h_2h_3+\beta_2h_1h_3+\beta_3h_1h_2$, then necessarily all the $\beta_i$'s are zero but one which is $1$, i.e. the class of $E$ in $A^2(F)$ is either $h_2h_3$ or $h_1h_3$ or $h_1h_2$.

Conversely we show that each line $E\subseteq F$ arises as the zero locus of a section of an indecomposable initialized aCM bundle of rank $2$ on $F$ with $c_1=0$. To this purpose, assume that the class of $E$ is $h_2h_3$. We know that $\omega_E\cong\cO_E(-2h)$, thus adjunction formula on $F$ implies $\det(\cN_{E\vert F})\cong\cO_E$. Theorem \ref{tSerre} with $\mathcal L:=\cO_F$ guarantees the existence of a vector bundle $\cE$ of rank $2$ fitting into a sequence of the form
$$
0\longrightarrow \cO_F\longrightarrow \cE\longrightarrow \cI_{E\vert F}\longrightarrow 0.
$$
Hence $h^1\big(F,\cE(th)\big)\le h^1\big(F, \cI_{E\vert F}(th)\big)$, $t\in \bZ$. The vanishing $h^1\big(\p7,\cI_{E\vert\p7}(th)\big)=h^2\big(\p7,\cI_{F\vert\p7}(th)\big)=0$ (recall that both $E$ and $F$ are aG), imply $h^1\big(F, \cI_{E\vert F}(th)\big)=0$, i.e. $h^1\big(F,\cE(th)\big)=0$. Since $c_1=0$, it follows that ${\cE}^\vee\cong\cE$, thus Serre's duality also yields $h^2\big(F,\cE(th)\big)=0$. We conclude that $\cE$ is an aCM bundle. It is trivial to check that $\cE$ is initialized. Finally, if $\cE$ were decomposable, then $\cE\cong\cM\oplus \cM^{-1}$ because $c_1=0$. Thus $h_2h_3=c_2=-c_1(\cM)^2$. But if $c_1(\cM)=\mu_1h_1+\mu_2h_2+\mu_3h_3$, then $c_1(\cM)^2=2(\mu_2\mu_3h_2h_3+\mu_1\mu_3h_1h_3+\mu_1\mu_2h_1h_2)$ which cannot coincide with $h_2h_3$.

Therefore we have proven the following theorem.

\begin{theorem}
  \label{tLine}
  There exist indecomposable initialized aCM bundles $\cE$ of rank $2$ with $c_1=0$. 

  Moreover the zero--locus of a general section of $\cE$ is a line and each line on $F$ can be obtained in this way.
\end{theorem}

Now we turn our attention to the case $c_1=2h$. Proposition \ref{pUpper6-1} guarantees that $\cE$ is Ulrich, hence globally generated. Thus also in this case the zero--locus  $E:=(s)_0$ of a general section $s\in H^0\big(F,\cE\big)$ is a smooth curve inside $F$. 

Using Sequence \eqref{seqIdeal} we obtain that $h^1\big(F,\cI_{E\vert F}\big)=h^1\big(F,\cI_{E\vert F}(h)\big)=0$, hence both $h^0\big(E,\cO_{E}\big)=1$ and $E$ is linearly normal (use Sequence \eqref{seqStandard}). We conclude that we can assume $E$ to be a smooth curve which is not contained in any hyperplane (from now on we will briefly refer to such a property by saying that $E$ is non--degenerate). Moreover, combining Lemma \ref{lInvariantE} and Formula \eqref{hc_2-6-1}, we obtain
$\deg(E)=hc_2=8$ and that $E$ is elliptic.

Conversely, let $E$ be a smooth, non--degenerate elliptic curve of degree $8$ on $F$. We know by adjunction that $\cO_E\cong\omega_E\cong\det(\cN_{E\vert F})\otimes\cO_F(-2h)$, hence $\det(\cN_{E\vert F})\cong\cO_F(2h)\otimes\cO_E$. The invertible sheaf $\cO_F(2h)$ satisfies the vanishing of the Theorem \ref{tSerre}, thus $E$ is the zero locus of a section $s$ of a vector bundle $\cE$ of rank $2$ on $F$ with $c_1=2h$ and the class of $E$ in $A^2(F)$ is $c_2$.

We conclude that in order to show the existence of indecomposable, initialized, aCM bundles $\cE$ of rank $2$ with $c_1=2h$, it suffices to prove the existence of smooth, non--degenerate elliptic curve of degree $8$ on $F$. The remaining part of this section is devoted to the proof of the existence of such curves.

Due to Proposition 1.1 and Corollary 2.2 of \cite{C--G--N}, we know that $E$ is aCM, then $h^1\big(\p7,\cI_{E\vert\p7}(th)\big)=0$. Since $F$ is aG, hence $h^2\big(\p7,\cI_{F\vert\p7}(th)\big)=0$, taking the cohomology of sequence
\begin{equation}
\label{seqIdealEF}
0\longrightarrow \cI_{F\vert\p7}\longrightarrow \cI_{E\vert\p7}\longrightarrow \cI_{E\vert F}\longrightarrow 0,
\end{equation}
it also follows that $h^1\big(F,\cI_{E\vert F}(th)\big)=0$. The cohomology of the Sequence \eqref{seqIdeal} yields $h^1\big(F,\cE(th)\big)=0$. Such a vanishing also implies  $h^2\big(F,\cE(th)\big)=0$ by Serre's duality. We conclude that $\cE$ is aCM. Thanks to Proposition \ref{pUpper6-1} we also know that $\cE$ is Ulrich. Thus non--degenerate elliptic curves on a del Pezzo threefold $F$ correspond to Ulrich bundles on $F$ with $c_1=2h$. 

First we deal with the possible values of $c_2$. 

\begin{lemma}
\label{lUlrichc_2}
If $\cE$ is an indecomposable initialized aCM bundle of rank $2$ with $c_1=2h$, then $c_2$ is either $2h_2h_3+2h_1h_3+4h_1h_2$ or $2h_2h_3+3h_1h_3+3h_1h_2$ up to permutation of the $h_i$'s.
\end{lemma}
\begin{proof}
The linear system $\vert h_i\vert$ on $F$ has dimension $1$. Let $D_i\in \vert h_i\vert$ be general. The cohomology of the exact sequence 
$$
0\longrightarrow \cO_F(h-h_i)\longrightarrow\cO_F(h)\longrightarrow\cO_{D_i}(h)\longrightarrow0
$$
yields $h^0\big(D_i,\cO_{D_i}(h)\big)=4$, thus $D_i$ spans a space of dimension $3$ in $\p7$. 

Let $E$ be the zero locus of a general section of $\cE$. Since $E$ is non--degenerate, we know that the restriction to $E$ of $\vert h_i\vert$ has dimension at least $1$. Since $E$ is elliptic, it follows that its degree, which is $\beta_i$, is greater than $2$, thanks to Riemann--Roch theorem on the curve $E$. Now the statement follows from the fact that $\deg(E)=hc_2=\beta_1+\beta_2+\beta_3=8$.
\end{proof}

We now proceed to construct explicitly such curves and bundles. To this purpose we will make use of degenerate elliptic curves. Since the proofs of the results below depend only on the fact that $F$ is a del Pezzo threefold, we will state and prove such results in general. Thus, in what follows, $F$ will be a del Pezzo threefold of degree $3\le d\le7$, embedded in $\p{d+1}$. Let $H$ be a general hyperplane in $\p{d+1}$, and define $S:=F\cap H$. 
The surface $S$ is smooth and connected of degree $d$. Moreover, adjunction on $F$ implies that $S$ is a del Pezzo surface. In particular we know that $S$ can be represented as the blow up of $\p2$ in $9-d$ points in general position. 

Thus the Picard group of $S$ is freely generated by the class $\ell$ of the pull--back of a general line in $\p2$ and by the classes of the $9-d$ exceptional divisors $e_1,\dots,e_{9-d}$ of the blow up.

Clearly we have a natural map $A^1(S)\to A^2(F)$ and we want to inspect which elliptic curves on $F$ come from curves on $S$. We will make such an analysis later on in the section. We now make some comments on the deformation theory of degenerate elliptic curves on a del Pezzo threefolds. 

\begin{proposition}
  \label{pNormal}
  Let $F$ be a del Pezzo threefold of degree $3\le d\le 7$. If $H\subseteq\p{d+1}$ is a general hyperplane and $C\subseteq S:=F\cap H$ is a smooth, connected elliptic curve of degree $\delta$, then 
  $$
  h^0\big(S,\cO_{S}(C)\big)=\delta+1,\qquad h^0\big(C,\cN_{C\vert F}\big)=2\delta,\qquad h^1\big(S,\cO_{S}(C)\big)=h^1\big(C,\cN_{C\vert F}\big)=0.
  $$
\end{proposition}
\begin{proof}
We have the isomorphisms $\cN_{C\vert S}\cong\cO_C(C)$ and $\cN_{S\vert F}\cong\cO_S(h)$. Thus we have the exact sequences
\begin{gather*}
0\longrightarrow \cO_C(C)\longrightarrow \cN_{C\vert F}\longrightarrow \cO_C(h)\longrightarrow 0,\\
0\longrightarrow \cO_S\longrightarrow \cO_S(C)\longrightarrow \cO_C(C)\longrightarrow0,
\end{gather*}
hence $\chi(\cN_{C\vert F})=\chi(\cO_C(C))+\chi(\cO_C(h))=\chi(\cO_S(C))+\chi(\cO_C(h))-1$. Since $p_a(C)=1$, it follows that  $h^1\big(C,\cO_{C}(h)\big)=0$, hence Riemann--Roch on $C$ implies $h^0\big(C,\cO_{C}(h)\big)=\delta$.

Let $a\ell-\sum_{i=1}^db_i e_i$ be the class of $C$ in $A^1(S)$. The conditions $\deg(C)=\delta$ and $p_a(C)=1$ yield $C^2=a^2-\sum_{i=1}^db_i^2=\delta>0$. Thus Riemann--Roch theorem on $S$ gives us $\chi(\cO_S(C))=\delta+1$, hence
\begin{equation}
\label{chiNormal}
\chi(\cN_{E\vert F})=2\delta.
\end{equation}

The surface $S$ is del Pezzo, then $h^i\big(S,\cO_S\big)=0$, $i=1,2$. It follows from the cohomology of the second sequence above and by Serre's duality on $S$ that 
$$
h^1\big(C,\cO_C(C)\big)=h^1\big(S,\cO_S(C)\big)=h^1\big(S,\cO_S(-C-h)\big).
$$
The last dimension is $0$. Indeed it suffices to take the cohomology of the second sequence above twisted by $\cO_S(-h)$. At this point the cohomology of the first sequence above yields $h^1\big(E,\cN_{E\vert F}\big)=0$. Combining this last vanishing with equality \eqref{chiNormal} we obtain the last equality of the statement.
\end{proof}

Now we come back to analyse the case $F\cong\p1\times\p1\times\p1$, identifying the elliptic curves on $F$ contained in a smooth hyperplane section. We use the same notation introduced at the end of the previous section: thus $H$ is a general hyperplane in $\p{7}$, and $S:=F\cap H$.

\begin{lemma}
  \label{lSurface}
  $C$ is a smooth, connected elliptic curve of degree $8$ on $S$ if and only if, up to automorphisms of $S$, $C$ is a smooth connected element in the class either  $3\ell-e_1$, or $ 4\ell-2e_1-2e_2$.
\end{lemma}
\begin{proof}
  Let $a\ell-\sum_{i=1}^{3}b_i e_i$ be the class of $C$ in $A^1(S)$. Since $C$ is effective, it follows that $a,b_i\ge0$. Moreover, one easily checks that the conditions $\deg(C)=8$ and $p_a(C)=1$ yield
  \begin{equation}
    \label{deg-genus}
    3a-\sum_{i=1}^3b_i=8,\qquad a^2-\sum_{i=1}^3b_i^2=8.
  \end{equation}
  Schwarz's inequality implies $(\sum_{i=1}^3b_i)^2\le 3\sum_{i=1}^3b_i^2$. Combining such an inequality with equalities \eqref{deg-genus}, we finally obtain 
  $6a^2-48a+88\le0$, i.e. $a=3,4,5$. 

  Now, again by equalities \eqref{deg-genus}, we obtain the cases  $3\ell-e_1$, $ 4\ell-2e_1-2e_2$, $5\ell-3e_1-2e_2-2e_3$. The automorphism of $S$ induced by the quadratic transformation of $\p2$ centered at the blown up points maps $(\ell,e_1,e_2,e_3)$ to $(2\ell-e_1-e_2-e_3, \ell-e_2-e_3,\ell-e_1-e_3,\ell-e_1-e_2)$, thus it
  transforms $5\ell-3e_1-2e_2-2e_3$ in $3\ell-e_1$. It follows that we can restrict our attention to  $3\ell-e_1$ and $ 4\ell-2e_1-2e_2$ only. 
\end{proof}

We recall that if $E$ is a non--degenerate, smooth elliptic curve on $F$, its class in $A^2(F)$ is either $2h_2h_3+3h_1h_3+3h_1h_2$ or $2h_2h_3+2h_1h_3+4h_1h_2$ (see Lemma \ref{lUlrichc_2}). We will show in the following the existence of non--degenerate, smooth, connected elliptic curves on $F$ of both types. To this purpose we will first show the existence of degenerate elliptic curves in both classes and we will postpone the proof of the existence of non--degenerate curves at the end of the present section. 

\begin{lemma}
  \label{lSurface-1}
  Let $H\subseteq\p{7}$ be a general hyperplane and let $C$ be a smooth, connected elliptic curve of degree $8$ on $S:=F\cap H$. The class of $C$ in $A^2(F)$ is either $2h_2h_3+3h_1h_3+3h_1h_2$ or $2h_2h_3+2h_1h_3+4h_1h_2$. 
\end{lemma}
\begin{proof}
  First we deal with the classes of $\ell$ and $e_i$ in $A^2(F)$. Recall that the class of $S$ inside $A^1(F)$ is the class of the hyperplane section, i.e. $h=h_1+h_2+h_3$: thus, e.g. $h_1S=h_1h=h_1h_2+h_1h_3$. Moreover the class of the hyperplane section of $S$ is $3\ell-e_1-e_2-e_3$ in $A^1(S)$. 

  We have $h_i h^2=2h_1h_2h_3$, thus $h_i S=h_i h$ is the class of a conic on $S$. Arguing as in the proof of Lemma \ref{lSurface}
 one easily checks that the classes of conics on $S$ are $\ell-e_i$. Thus we can assume that $h_i S=\ell-e_i$ in $A^1(S)$. 

  Now, we turn our attention to $\ell$. Let  $\gamma_1h_2h_3+\gamma_2h_1h_3+\gamma_3h_1h_2$ its class in $A^2(F)$. We have $\gamma_i=\ell h_i\ge0$ and 
  $$
  \gamma_1+\gamma_2+\gamma_3=\ell h=\deg(\ell)=\ell(3\ell-e_1-e_2-e_3)=3
  $$
  If $\gamma_1=0$, then $e_1 h_3=(\ell-h_1S)h_3=-h_1h_2h_3$. Since $e_1$ and $h_3$ are both effective, we obtain an absurd. Thus $\gamma_1\ge1$ and, similarly, $\gamma_2,\gamma_3\ge1$. We conclude that the classes of $\ell$, $e_1$, $e_2$, $e_3$ in $A^2(F)$ are $h_2h_3+h_1h_3+h_1h_2$, $h_2h_3$, $h_1h_3$, $h_1h_2$ respectively.

  Thanks to Lemma \ref{lSurface}, the assertion about the class of $C$ in $A^2(F)$ now follows from direct substitution.
\end{proof}

As explained at the beginning of our analysis, the curves $C$ described in Lemma \ref{lSurface-1}, being degenerate, do not correspond to the bundles we are interested in: indeed such bundles correspond to non--degenerate curves. Thus we have finally to show the existence of non--degenerate curves in the same classes. We will check this with a well--known deformation argument (see \cite{Fa2}). 

\begin{proposition}
  \label{pDef6-1}
  If $E\subseteq F$ is a general element in the class either $2h_2h_3+3h_1h_3+3h_1h_2$ or $2h_2h_3+2h_1h_3+4h_1h_2$. Then $E$ is a non--degenerate smooth, connected elliptic curve.
\end{proposition}
\begin{proof}
  Let $\mathcal H$ be the union of the components of $\Hilb^{8t}(F)$ containing non--degenerate smooth connected elliptic curves of degree $8$. Take the incidence variety 
  $$
  \mathcal X:=\{\ (C,H)\in \mathcal H\times({\p7})^\vee\ \vert\ C\subseteq H\ \}
  $$
  Let $(C,H)\in\mathcal X$ be general. In particular $H\subseteq\p7$ is general, then $S:=F\cap H$ is a smooth del Pezzo surface of degree $6$ as above and $C\subseteq S$ is a smooth connected elliptic curve whose class in $A^2(F)$ is $\gamma_1h_2h_3+\gamma_2h_1h_3+\gamma_3h_1h_2$. Anyhow $\deg(C)=8$, hence Proposition \ref{pNormal} yields $h^0\big(C,\cN_{C\vert S}\big)=16$ and $h^1\big(C,\cN_{C\vert S}\big)=0$. It follows that $\mathcal H$ is smooth at the point corresponding to $C$ (briefly we say that $C$ is unobstructed) and it has dimension $16$.  Assume that all deformation of $C$ inside $F$ were degenerate.

  On the one hand, the image of the projection $\mathcal X\to\mathcal H$ would contain an open neighbourhood of the point corresponding to $C$. Since the points in the fibre over $C$ are parametrized by the hyperplanes containing $C$, it follows that such a projection would be generically injective, hence $\mathcal X$ would contain an irreducible component $\widehat{\mathcal X}$ of dimension $16$ containing $(C,H)$.

  On the other hand the fibre of the projection $\widehat{\mathcal X}\to({\p7})^\vee$ over $H$ is isomorphic to the projectivization of $H^0\big(S,\cO_S(C)\big)$ which is $\p8$ (see Proposition \ref{pNormal}). We conclude that $\dim(\widehat{\mathcal X})\le 15$, a contradiction. We conclude that each general deformation of $C$ inside $F$ is non--degenerate in $\p7$. 
  
  Let $\mathcal C\subseteq F\times B\to B$ a flat family of curves in $\mathcal H$ with special fibre ${\mathcal C}_{b_0}\cong C$ over $b_0$. Since $C$ is unobstructed, we can assume that $B$ is integral. Since $\mathcal C\to B$ is flat and $C$ is integral, it follows that $\mathcal C$ is integral too. Take a general element in the class $h_i$, say $Q_i$, and consider the family  ${\mathcal Q}_i:={\mathcal C}\cap (Q_i\times B)\to B$. Since $Q_i$ is general and $h_i$ is globally generated, we can assume that ${\mathcal Q}_i$ is a family of $0$--cycles of $F$. Up to a proper choice of homogeneous coordinates  $t_0^{(i)},t_1^{(i)}$ on the $i^{th}$ copy of $\p1$ inside the product $F$, we can assume $\cO_{{\mathcal Q}_i}\cong\cO_{\mathcal C}/(t_1^{(i)})$. Thanks to the choice of $Q_i$, the element $t_1^{(i)}$ is regular element in $\cO_{\mathcal C}$. Thus the Corollary of Theorem 22.5 of \cite{Ma}, implies that ${\mathcal Q}_i$ is flat over $B$. By semicontinuity the degree of the fibre of ${\mathcal Q}_i$ over $b$, which is ${\mathcal C}_b h_i$, is upper semicontinuous, thus there exists an open subset $B_i\subseteq B$ containing $b_0$ such that ${\mathcal C}_b h_i\le \gamma_i$. Since 
  $$
  \gamma_1+\gamma_2+\gamma_3=8=\deg({\mathcal C}_b)={\mathcal C}_b h=\sum_{i=1}^3{\mathcal C}_b h_i\le \gamma_1+\gamma_2+\gamma_3,
  $$
  at each $b\in B':=B_1\cap B_2\cap B_3\subseteq B$, we finally conclude that ${\mathcal C}_b h_i= \gamma_i$, i.e. ${\mathcal C}_b$ is in the same class of $C$ inside $A^2(F)$.  Hence each general element $E\subseteq F$ in the classes $2h_2h_3+3h_1h_3+3h_1h_2$ and $2h_2h_3+2h_1h_3+4h_1h_2$ is non--degenerate.
\end{proof}

We finally state our existence result.

\begin{theorem}
  \label{tElliptic}
  There exist indecomposable initialized aCM bundles $\cE$ of rank $2$ with $c_1=2h$ and $c_2$ either $2h_2h_3+3h_1h_3+3h_1h_2$ or $2h_2h_3+2h_1h_3+4h_1h_2$ up to permutations of the $h_i$'s. Moreover the zero--locus of a general section of $\cE$ is an elliptic normal curve. 
  
Conversely,  up to permutations of the $h_i$'s, each elliptic normal curve on $F$ can be obtained as the zero locus of a general section of an initialized aCM bundles $\cE$ of rank $2$ with $c_1=2h$ and $c_2$ either $2h_2h_3+3h_1h_3+3h_1h_2$ or $2h_2h_3+2h_1h_3+4h_1h_2$. The bundle $\cE$ is indecomposable if and only if it is not the complete intersection of two divisors $Q_1\in\vert 2h_2+h_3\vert$ and $Q_2\in\vert 2h_1+h_3\vert$. 
\end{theorem}
\begin{proof}
Almost all the statement follows from the previous proposition and from the correspondence between non--degenerate smooth elliptic curves and bundles with $c_1=2h$. The only unproved assertion is that there are indecomposable bundles of each type.

Assume that $\cE\cong{\mathcal L}_1\oplus{\mathcal L}_2$ for some ${\mathcal L}_i\in\Pic(F)$. Thus we know that ${\mathcal L}_i$ is aCM. Moreover either both the ${\mathcal L}_i$ are initialized or one of them is initialized and the other has no sections. 

Thanks to Lemma \ref{lInvaCM} and to the equalities
\begin{gather*}
  h^0\big(F,{\mathcal L}_1\big)+h^0\big(F,{\mathcal L}_2\big)=h^0\big(F,\cE\big)=12,\\
  c_1({\mathcal L}_1)+c_1({\mathcal L}_1)=c_1=2h
\end{gather*}
it is easy to check that $\cE\cong\cO_F(2h_2+h_3)\oplus\cO_F(2h_1+h_3)$ necessarily. In particular in the case $(2,3,3)$, i.e. $c_2=2h_2h_3+3h_1h_3+3h_1h_2$, the vector bundle $\cE$ is indecomposable.

Consider the other case $(2,2,4)$, i.e. $c_2=2h_2h_3+2h_1h_3+4h_1h_2$. We have 
$$
\dim_k\left(\Ext^1_F\big(\cO_F(2h_2+h_3),\cO_F(2h_1+h_3)\big)\right)=h^1\big(F,\cO_F(2h_1-2h_2)\big)=3,
$$
thus there are non--trivial extensions of the form
\begin{equation}
\label{seqExtension}
0\longrightarrow\cO_F(2h_1+h_3)\longrightarrow\cE\longrightarrow\cO_F(2h_2+h_3)\longrightarrow0.
\end{equation}
Notice that $\cE$ is aCM, because both $\cO_F(2h_2+h_3)$ and $\cO_F(2h_1+h_3)$ are aCM.

If $\cE$ is decomposable, then we checked above that $\cE\cong\cO_F(2h_2+h_3)\oplus\cO_F(2h_1+h_3)$. Since 
$$
\dim\left(\Hom_F\big(\cO_F(2h_1+h_3),\cO_F(2h_2+h_3)\big)\right)=h^0\big(F,\cE(-h_1+2h_2-h_3)\big)=0,
$$
it follows that the epimorphism $\cE\to\cO_F(2h_2+h_3)$ has a section, thus Sequence \eqref{seqExtension} splits. In particular $\cE$ splits if and only if the extension \eqref{seqExtension} splits too. Thus each non--zero section of $\Ext^1_F\big(\cO_F(2h_2+h_3),\cO_F(2h_1+h_3)\big)$ induces an indecomposable bundle.
\end{proof}

\section{The sporadic extremal case}
\label{sSporadic}
In this section we deal with the case $c_1=h_1+2h_2+3h_3$. Again Proposition \ref{pUpper6-1} guarantees that $\cE$ is Ulrich, hence globally generated. It follows that the zero--locus  $E:=(s)_0$ of a general section $s\in H^0\big(F,\cE\big)$ is a smooth curve inside $F$. 

Notice that $\cO_F(-c_1)$ is aCM. This fact and the cohomology of Sequence \eqref{seqIdeal} yield $h^1\big(F,\cI_{E\vert F}\big)=h^1\big(F,\cI_{E\vert F}(h)\big)=0$. It follows that both $h^0\big(E,\cO_{E}\big)=1$ and $E$ is linearly normal (use Sequence \eqref{seqStandard}). We conclude that we can assume $E$ to be a non--degenerate, smooth, irreducible curve. Moreover, again by combining Lemma \ref{lInvariantE} with $D=0$ and Formula \eqref{hc_2-6-1}, one easily obtains $\deg(E)=hc_2=7$. It follows that $E$ is a rational normal curve in $\p7$.

Conversely, let $E$ be a non--degenerate, smooth, rational curve of degree $7$ in the class $\beta_1h_2h_3+\beta_2h_1h_3+\beta_3h_1h_2$ on $A^2(F)$. We know by adjunction that $\cO_\p1(-2)\cong\omega_E\cong\det(\cN_{E\vert F})\otimes\cO_F(-2h)$, thus $\det(\cN_{E\vert F})\cong\cO_{\p1}(12)$. The restriction to $E$ of $\vert h_i\vert$, $i=1,2,3$, is a linear system of divisors on $\p1$ of degree $\beta_i$, hence its dimension is $\beta_i$. Since $E$ is non--degenerate and $\vert h_i\vert$ is a pencil on $F$, it follows that $\beta_i\ge1$. The same argument applied to $\vert h_i+h_j\vert$, $i,j=1,2,3$ and $i\ne j$, similarly shows that $\beta_i+\beta_j\ge3$. We conclude that the possible values of  $(\beta_1,\beta_2,\beta_3)$ are  $(3,2,2)$, $(4,1,2)$, $(3,3,1)$ up to permutations of the $h_i$'s.

Let us now consider the first case, i.e. the class of $E$ is $3h_2h_3+2h_1h_3+2h_1h_2$. Such an $E$ cannot be the zero locus of a section of any aCM bundle $\cE$ with $c_2=3h_2h_3+2h_1h_3+2h_1h_2$. On the one hand, Lemma  \ref{lInvariantE} would imply that $c_1c_2=12$. It follows that if $c_1=\alpha_1h_1+\alpha_2h_2+\alpha_3h_3$, then $(\alpha_1,\alpha_2,\alpha_3)=(1,2,3)$, up to permutations. On the other hand, the first of the equalities of System \eqref{AlphaBeta} shows that there is no permutation of the $h_i$'s such that $c_1c_2=12$.

\begin{remark} 
Smooth rational curves in the class $3h_2h_3+2h_1h_3+2h_1h_2$ actually exist. Indeed choose three general maps from $\p1$ to $\p1$ of degrees $3$, $2$, $2$. Their product gives a morphism $\p1\to F$ whose image $E$ is a rational curve.

Using any computer algebra software (e.g. \cocoa: see \cite{CoCoA}) one can compute the homogeneous ideal of $E$ in $\p7$ checking that it is minimally generated by quadratic forms, hence $E$ is non--degenerate. In the same way one can also check that the Hilbert polynomial of $E$ is $7t + 1$. It follows that the arithmetic genus of $E$ is $0$. We conclude that $E$ is a smooth rational curve of degree $7$.

The sheaf $\mathcal L:=\cO_F (2h_1 + 2h_2 + h_3)$ satisfies the hypothesis of Theorem \ref{tSerre}. Thus $E$ is the zero locus of a section of rank $2$ bundle $\cE$ on $F$ with $c_1 = 2h_1 +2h_2 +h_3$ and $c_2 = 3h_2h_3 +2h_1h_3 +2h_1h_2$. The cohomology of Sequence \eqref{seqIdeal} twisted by $\cO_F(-h)$ shows that $\cE$ is initialized. An analysis of all the possible cases also shows that $\cE$ is indecomposable. Nevertheless, $\cE$ cannot be aCM, thanks to the discussion above.
\end{remark}

Let us examine the cases $(4,1,2)$ and $(3,3,1)$. The sheaf $\mathcal L:=\cO_F(h_1+2h_2+3h_3)$ satisfies the hypothesis of Theorem \ref{tSerre}. Thus $E$ is the zero locus of a section $s$ of a vector bundle $\cE$ of rank $2$ on $F$ with $c_1=h_1+2h_2+3h_3$ and $c_2$ either $4h_2h_3+h_1h_3+2h_1h_2$, or $3h_2h_3+3h_1h_3+h_1h_2$ respectively. Moreover, $E$ is aCM (see Proposition 1.1 and Corollary 2.2 of \cite{C--G--N}), and the same holds for $F$. 

\begin{lemma}
Let $E$ be a non--degenerate, smooth, rational curve of degree $7$ whose class in $A^2(F)$ is either $4h_2h_3+h_1h_3+2h_1h_2$, or $3h_2h_3+3h_1h_3+h_1h_2$. Then the bundle $\cE$ defined above is Ulrich.
\end{lemma}
\begin{proof}
Since $E$ is aCM, it follows that $h^1\big(\p7,\cI_{E\vert\p7}(th)\big)=h^2\big(\p7,\cI_{F\vert\p7}(th)\big)=0$. Again the cohomology of Sequence \eqref{seqIdealEF} implies $h^1\big(F,\cI_{E\vert F}(th)\big)=0$. The cohomology of Sequence \eqref{seqIdeal} twisted by $\cO_F(-c_1)$ implies for $t\in\bZ$
$$
h^2\big(F,\cE(th)\big)=h^1\big(F,\cE^\vee((-2-t)h)\big)=h^1\big(F,\cI_{E\vert F}((-2-t)h)\big)=0.
$$

We check that $h^1(F,{\cE}(th)\big)=h^1\big(F,\cI_{E\vert F}(c_1+th)\big)=0$, $t\in\bZ$. Notice that $\deg(\cO_E(c_1+th))=(c_1+th)c_2=12+7t$. Hence $h^1\big(F,\cI_{E\vert F}(c_1+th)\big)\le h^0\big(E,\cO_{E}(c_1+th)\big) =0$ when $t\le -2$, because $E\cong\p1$. 

Let $t=-1$. We have $(h_2+2h_3)E=5$. Hence $h^0\big(E,\cO_E(h_2+2h_3)\big)=6$. The cohomologies of Sequences \eqref{seqStandard} and \eqref{seqIdeal} suitably twisted yield
$$
h^1\big(F,\cE(-h))=h^0\big(F,\cE(-h))=h^0\big(F,\cI_{E\vert F}(h_2+2h_3)\big).
$$

Assume the existence of a surface $M\subseteq F$ in $\vert\cO_F(h_2+2h_3)\vert$ through $E$. We have $\deg(M)=6$. If $M$ were not integral, then $E$, being integral, would be contained in an integral surface $M'\subseteq F$ of degree at most $5$. In particular $M'$ would be degenerate in $\p7$, thus the same would hold for $E$, a contradiction. It follows that $M$ is integral and it is easy to check that it is non--degenerate, thus a surface of minimal degree in $\p7$. Hence it is either a smooth rational normal scroll, or it is the cone onto a smooth rational normal curve of degree $6$ contained in a hyperplane $H$ with vertex a point $V\not\in H$ (see \cite{G--H}, p. 522).

In second case the tangent space at $M$ in $V$ is the whole $\p7$. Since such a space is contained in the tangent space at $F$ in $V$, which is a smooth variety, we get a contradiction. Hence $M\cong\bP(\cO_{\p1}\oplus\cO_{\p1}(-e))$ for some integer $e\ge0$. Let $\xi$ be the tautological divisor on $M$ and $f$ the class of a fibre. The embedding $M\subseteq\p7$ is given by some linear system $\vert\xi+af\vert$ with $a>0$.

Since a rational normal curve has no trisecant curves because it is cut out by quadrics (see \cite{G--H}, p. 530), it follows that $E$ is either in $\vert (\xi+af)+bf\vert$, or in $\vert 2(\xi+af)-cf\vert$ for suitable integers $b$ and $c$. The conditions $\deg(M)=6$ and $\deg(E)=7$ imply $b=1$ and $c=-5$. 

The class of $\vert\xi+af\vert$ in $A^2(F)$ is $hM=3h_2h_3+2h_1h_3+h_1h_2$. The linear system $\vert\cO_M(h_3)\vert$ is a pencil of lines on $M$. An easy computation in $\Pic(M)$ shows that the only pencil of lines on $M$ is $\vert f\vert$, thus the class of $f$ inside $A^2(F)$ is $h_2h_3$.  By combining the above results we finally obtain that the class of $E$ in $A^2(F)$ must be either $4h_2h_3+2h_1h_3+h_1h_2$, or $h_2h_3+4h_1h_3+2h_1h_2$, contradicting the hypothesis on $E$. We conclude that $h^1\big(F,\cE(-h))=h^0\big(F,\cE(-h))=0$. In particular $\cE$ is initialized.

Let us examine the case $t\ge0$. In this case we have a commutative diagram
$$
\begin{CD}
  H^0\big(F,\cO_{F}((t+1)h)\big)\otimes H^0\big(F,\cO_{F}(M)\big)@>\varphi_t\otimes \vartheta_{-1}>> H^0\big(E,\cO_E((t+1)h)\big)\otimes H^0\big(E,\cO_E(M)\big)\\
  @VVV @V\vartheta_tVV\\
  H^0\big(F,\cO_{F}(c_1+th)\big)@>>> H^0\big(E,\cO_E(c_1+th)\big)\\
\end{CD}
$$
where $\varphi_t$ is the natural restriction map: it follows that the top horizontal arrow is surjective. Since $E\cong\p1$ the right vertical arrow is an epimorphism. Thus $\vartheta_t$ is surjective too, hence $h^1\big(F,\cI_{E\vert F}(c_1+th)\big)=0$, $t\ge0$. 
The vanishing of $h^i(F,{\cE}(th)\big)=0$, $i=1,2$ and $t\in\bZ$, proved above implies that $\cE$ is aCM, hence $\cE$ is an Ulrich bundle thanks to Proposition \ref{pUpper6-1}. 
\end{proof}

We now construct examples of Ulrich bundles with $c_1=h_1+2h_2+3h_3$. We have that
$$
\dim_k\left(\Ext^1_F\big(\cO_F(2h_2+h_3),\cO_F(h_1+2h_3)\big)\right)=h^1\big(F,\cO_F(h_1-2h_2+h_3)\big)=4,
$$
thus there are non--trivial extensions of the form
$$
  0\longrightarrow \cO_F(h_1+2h_3)\longrightarrow \cE\longrightarrow  \cO_F(2h_2+h_3)\longrightarrow 0.
$$
The bundle $\cE$  has Chern classes are $c_1=h_1+2h_2+3h_3$ and $c_2=4h_2h_3+h_1h_3+2h_1h_2$. 

Since both $\cO_F(2h_2+h_3)$ and $\cO_F(h_1+2h_3)$ are aCM, looking at the cohomology of above sequence twisted by $\cO_F(th)$, it follows that the same is true for $\cE$. Moreover, the same cohomology sequence and vanishing also yield that $\cE$ is initialized.
Arguing as in the previous section we are also able to show that $\cE$ is indecomposable, with the exception of the trivial extension.

The above discussion partially proves the following result.

\begin{theorem}
  \label{tRational}
  There exist indecomposable initialized aCM bundles $\cE$ of rank $2$ with $c_1=h_1+2h_2+3h_3$ and $c_2$ either $4h_2h_3+h_1h_3+2h_1h_2$ or $3h_2h_3+3h_1h_3+h_1h_2$ up to permutation of the $h_i$'s. Moreover the zero--locus of a general section of $\cE$ is a rational normal curve.
  
Conversely, up to permutations of the $h_i$'s, each rational normal curve on $F$ whose class in $A^2(F)$ is not $3h_2h_3+2h_1h_3+2h_1h_2$ can be obtained as the zero locus of a general section of an initialized aCM bundles $\cE$ of rank $2$ with $c_1=h_1+2h_2+3h_3$ and $c_2$ either $4h_2h_3+h_1h_3+2h_1h_2$ or $3h_2h_3+3h_1h_3+h_1h_2$. The bundle $\cE$ is indecomposable if and only if it is not the complete intersection of two divisors $Q_1\in\vert 2h_2+h_3\vert$ and $Q_2\in\vert h_1+2h_3\vert$. 
\end{theorem}
\begin{proof}
  The above discussion proves the existence of of a $3$--dimensional family of initialized aCM vector bundles with $c_2=4h_2h_3+h_1h_3+2h_1h_2$. The only decomposable bundle in such a family is $\cO_F(h_1+2h_3)\oplus\cO_F(2h_2+h_3)$, whose sections are trivially complete intersection curves inside $F$.
  
  Now we deal with the case $c_2=3h_2h_3+3h_1h_3+h_1h_2$. Again the above discussion shows that if such a vector bundle exists, then it is indecomposable. In order to show its existence we again work by deforming suitable curves on a hyperplane section of $F$ as we did in the previous section. Let $S:=H\cap F$ be a general hyperplane section of $F$. $S$ is a del Pezzo surface of degree $6$ in $\p6$. We know by the proof of Lemma \ref{lSurface-1} that the classes of $\ell$, $e_1$, $e_2$, $e_3$ in $A^2(F)$ are $h_2h_3+h_1h_3+h_1h_2$, $h_2h_3$, $h_1h_3$, $h_1h_2$ respectively. Using the same notations of Lemma \ref{lSurface} one easily checks that $S$ contains rational curves of degree $7$ and their classes are $3\ell-2e_3$, $4\ell-e_1-e_2-3e_3$ which are both in the class of $3h_2h_3+3h_1h_3+h_1h_2$.

  Take a smooth rational curve $C\subseteq S$ of degree $7$. We have $\cN_{C\vert S}\cong\cO_C(C)$ and $\cN_{S\vert F}\cong\cO_S(h S)$, thus the exact sequences
  \begin{gather*}
    0\longrightarrow \cO_C(C)\longrightarrow \cN_{C\vert F}\longrightarrow \cO_C(h)\longrightarrow 0,\\
    0\longrightarrow \cO_S\longrightarrow \cO_S(C)\longrightarrow \cO_C(C)\longrightarrow0,
  \end{gather*}
  Since $C\cong\p1$, it follows that $\cO_C(C)\cong\cO_{\p1}(5)$, $\cO_C(h)\cong\cO_{\p1}(7)$, thus $h^0\big(S,\cO_{S}(C)\big)=7$, $h^0\big(C,\cN_{E\vert F}\big)=14$,  $h^1\big(C,\cN_{C\vert F}\big)=0$. 

  Arguing as in the proof of Proposition \ref{pDef6-1} one checks the existence of a non--degenerate smooth rational curve $E$ whose class is  $3h_2h_3+3h_1h_3+h_1h_2$.
\end{proof}

\section{The intermediate cases}
\label{sIntermediate}
In this section we will show which other indecomposable, initialized, aCM bundles $\cE$ of rank $2$ are actually admissible besides the ones occurring in the extremal cases. We know that for the general section $s\in H^0\big(F,\cE\big)$ its zero--locus $E:=(s)_0$ has pure dimension $1$.

System \eqref{AlphaBeta}, Inequality \eqref{positivity} with $D=0$ and Formulas \eqref{c_1c_2-6-1}, \eqref{hc_2-6-1}  yield the following list (up to permutations). 
$$
\begin{array}{|c|c|c|c|c|}           \hline
  \text{Case} &{ (\alpha_1,\alpha_2,\alpha_3) } & {(\beta_1,\beta_2,\beta_3) } & { \deg(E) } & { p_a(E) } \\ \hline
  L&(0,0,2) & (0,0,0) & 0 & -1\\  \hline
  M&(0,0,1) & (1,0,0) & 1 & 0\\  \hline
  N&(0,1,1) & (1,0,0) & 1 & 0\\  \hline
  P&(0,1,2) & (1,0,0) & 1 & 0\\  \hline
  Q&(1,1,1) & (1,1,0) & 2 & 0\\  \hline
  R&(1,1,1) & (2,0,0) & 2 & 0\\  \hline
  S&(0,2,2) & (2,0,0) & 2 & -1\\  \hline
  T&(1,1,2) & (1,1,1) & 3 & 0\\  \hline
  U&(1,1,2) & (0,2,1) & 3 & 0\\  \hline
  V&(1,2,2) & (2,2,1) & 5 & 0\\  \hline
  W&(1,2,2) & (2,3,0) & 5 & 0\\  \hline
\end{array}
$$
\vskip0.5truecm
\noindent It is immediate to exclude the case L because $E\ne\emptyset$ due to Lemma \ref{lNonEmpty6-1}.

Let $\cE$ be in the list above. Then $\cE^\vee(th)$ is an indecomposable aCM bundle too. We take $t$ in such a way that it is also initialized. Thus $0\le 2t-\alpha_1\le2$ (see Proposition \ref{pUpper6-1}), hence $t=1$. We have 
\begin{gather*}
  c_1(\cE^\vee(h))=(2-\alpha_1)h_1+(2-\alpha_2)h_2+(2-\alpha_3)h_3,\\
  {\begin{align*}
      c_2(\cE^\vee(h))&=(\beta_1+2-(\alpha_2+\alpha_3))h_2h_3+\\
      &+(\beta_2+2-(\alpha_1+\alpha_3))h_1h_3+(\beta_3+2-(\alpha_1+\alpha_2))h_1h_2
    \end{align*}}
\end{gather*}
In case S, then $c_2(\cE^\vee(h))=0$, thus we can again exclude it. In case U (resp. W), then $c_2(\cE^\vee(h))=-h_2h_3+h_1h_3+h_1h_2$ (resp. $2h_1h_3-h_1h_2$). Due to the positivity of the $\beta_i$'s (see Formula \eqref{positivity}) also these two cases cannot occur. The same argument shows that cases T and V occur if and only if cases N and M occur respectively.

We now examine the remaining cases one by one, checking which of them actually occurs. We have to analyze cases M, N, P, Q, R.

\subsection{The case M (and V)}
In case M we know by the table that $E$ must be a line. 

Let $E\subseteq F$ be a line in the class $h_2h_3\in A^2(F)$. Since $\omega_E\cong\cO_E(-2h)$, it follows that $\det(\cN_{E\vert F})\cong\cO_E$ which 
thus coincides with $\cO_F(h_3)\otimes\cO_E$. Taking into account the vanishing $h^2\big(F,\cO_F(-h_3)\big)=0$, we thus know the existence of an exact sequence 
$$
0\longrightarrow \cO_F\longrightarrow \cE\longrightarrow\cI_{E\vert F}(h_3)\longrightarrow0.
$$
where $\cE$ is a vector bundle of rank $2$ on $F$ (see Theorem \ref{tSerre}). It is immediate to check that $\cE$ is initialized. We now check that $\cE$ is aCM. Trivially we have $h^1\big(F,\cE(th)\big)=h^1\big(F,\cI_{E\vert F}(th+h_3)\big)$.

The cohomology of Sequence \eqref{seqStandard} twisted by $\cO_F(th+h_3)$ gives rise to the exact sequence
$$
  H^0\big(F,\cO_F(th+h_3)\big)\longrightarrow H^0\big(E,\cO_E(th+h_3)\big)\longrightarrow H^1\big(F,\cI_{E\vert F}(th+h_3)\big)\longrightarrow0.
$$
The multiplication by $h_3$ gives rise to the commutative diagram
$$
\begin{CD}
  H^0\big(F,\cO_{F}(th)\big)@>\varphi_t>> H^0\big(E,\cO_E(th)\big)\\
  @VVV @V\psi_t VV\\
  H^0\big(F,\cO_{F}(th+h_3)\big)@>>> H^0\big(E,\cO_E(th+h_3)\big)\\
\end{CD}
$$
Since both $E$ and $F$ are aG, it is easy to check that $\varphi_t$ is surjective (see the proof of the analogous fact in Section \ref{sExtremal}). Moreover $E$ is in the class of $h_2h_3$, thus $\psi_t$ is an isomorphism. It follows that the map below is surjective, hence $h^1\big(F,\cI_{E\vert F}(th+h_3)\big)=0$.

Now we look at the vanishing of $H^2$. We know that we can also write the exact sequence
$$
0\longrightarrow \cO_F(-h_3)\longrightarrow \cE^\vee\longrightarrow\cI_{E\vert F}\longrightarrow0.
$$
We thus obtain
$$
h^2\big(F,\cE(th)\big)=h^1\big(F,\cE^\vee(-(2+t)h)\big)=h^1\big(F,\cI_{E\vert F}(-(2+t)h)\big).
$$ 
But the last number is the dimension of the cokernel of the map $\varphi_{-(2+t)}$ which is zero. We conclude that $\cE$ is aCM.

Assume that $\cE$ is decomposable. Thus it is the sum of two aCM invertible sheaves, one of them being initialized, the other one either initialized or without sections. Such sheaves are listed in Lemma \ref{lInvaCM}. Looking at $c_1$ the unique possibilities are then $\cO_F\oplus\cO_F(h_3)$, $\cO_F(-h_2)\oplus\cO_F(h_2+h_3)$, $\cO_F(-h_2-h_3)\oplus\cO_F(h_2+2h_3)$. But, in these cases, $c_2$ is $0$, $-h_2h_3$, $-3h_2h_3$ respectively. We conclude that $\cE$ is indecomposable.

In particular both the cases M and V are admissible.

\subsection{The case N (and T)}
In case N we know by the table that $E\subseteq F$ is a line in the class $h_2h_3\in A^2(F)$.  On the one hand we know that $\det({\cN_{E\vert F}})\cong\cO_F(h_2+h_3)\otimes\cO_E$. Since $h^2\big(F,\cO_F(-h_2-h_3)\big)=0$, it follows the existence of an exact sequence 
$$
0\longrightarrow \cO_F(-h_2-h_3)\longrightarrow \cE^\vee\longrightarrow\cI_{E\vert F}\longrightarrow0.
$$
where $\cE$ is a vector bundle of rank $2$ on $F$. Let $i=2,3$: there are $D_i\in\vert h_i\vert$ containing $E$, because $c_2h_i=0$. Thus $E=D_2\cap D_3$ necessarily. The associated Koszul complex is the exact sequence 
$$
0\longrightarrow \cO_F(-h_2-h_3)\longrightarrow \cO_F(-h_2)\oplus\cO_F(-h_3)\longrightarrow\cI_{E\vert F}\longrightarrow0.
$$

We have $h^1\big(F, \cO_F(-h_2-h_3)\big)=0$ (see Lemma \ref{lInvaCM}), hence the two sequences above must be isomorphic due to Theorem \ref{tSerre}. It follows that  $\cE\cong\cO_F(h_2)\oplus\cO_F(h_3)$. 

In particular both the cases N and T are not admissible in the sense that the sheaf $\cE$ exists, but it is decomposable.

\subsection{The case P}
Arguing as in the case N we check that $\cE\cong\cO_F(h_2+h_3)\oplus\cO_F(h_3)$. Hence this case is not admissible in the sense that the sheaf $\cE$ exists, but it is decomposable.

\subsection{The cases Q and R}
In these cases $E$ is a curve of degree $2$ of genus $0$. $E$ is either irreducible or the union of two concurrent lines, because $p_a(E)=0$, or a double line. 

In case Q only the first two cases are possible and we can easily construct the exact sequence 
$$
0\longrightarrow \cO_F(-h)\longrightarrow \cE^\vee\longrightarrow\cI_{E\vert F}\longrightarrow0
$$
where $\cE$ is a vector bundle of rank $2$ on $F$. 

We have $h_3c_2=0$: since $\vert h_3\vert$ is base--point--free, it follows the existence of $Q\in\vert h_3\vert$ containing $E$. $Q$ is a smooth quadric surface, $E$ is cut out on $Q$ by a hyperplane. Since $\cO_F(h)\otimes\cO_Q\cong\cO_F(h_1+h_3)\otimes\cO_Q$, it follows the existence of $D\in\vert h_1+h_2\vert$ such that $D\cap Q=E$. In particular we have a corresponding Koszul complex
$$
0\longrightarrow \cO_F(-h)\longrightarrow \cO_F(-h_1-h_2)\oplus\cO_F(-h_3)\longrightarrow\cI_{E\vert F}\longrightarrow0.
$$
As in the previous cases N and P, we deduce that $\cE\cong\cO_F(h_1+h_2)\oplus\cO_F(h_3)$. 

Let us examine case R. We claim that $E$ is not reduced. E.g., assume $E$ is the union of two distinct concurrent lines, say $L,M$, and let $p:=L\cap M$ be their intersection point. Thus there are $Q_i\in\vert h_i\vert$, $i=1,2$ such that $p\in Q_1\cap Q_2$. Since $LQ_i+MQ_i= E Q_i=0$, it would follow that $E=L\cup M\subseteq Q_1\cap Q_2$, hence its class inside $A^2(F)$ would be $h_2h_3$. But we know that its class is $c_2=2h_2h_3$, a contradiction. If $E$ is irreducible we can argue similarly.

Thus $E$ is a double structure on a line $E_{red}$ whose class in $A^2(F)$ is $h_2h_3$ necessarily. The general theory of double structures (see e.g. \cite{B--E}) gives us an exact sequence of the form
$$
0\longrightarrow \cC_{E_{red}\vert E}\longrightarrow \cO_E\longrightarrow \cO_{E_{red}}\longrightarrow 0.
$$
The conormal sheaf $\cC_{E_{red}\vert E}$ is an invertible sheaf on $E_{red}\cong\p1$, thus $\cC_{E_{red}\vert E}\cong\cO_{\p1}(-a)$. Moreover $a=1$, because $p_a(E)=0$ (see Section 2 of \cite{B--E}). 

Recall that the Hilbert scheme $\Gamma$ of lines in $F$ is the union of three components $\Gamma_i$, each of them isomorphic to $\p1\times\p1$ (Proposition 3.5.6 of \cite{I--P}). 

On the one hand, all the lines on $F$ in the class $h_2h_3\in A^2(F)$ are parameterized by the same copy of $\p1\times\p1$ and we denote it by $\Gamma_1$. 
The variety $F$ does not contain planes, hence the universal family $S_1\to\Gamma_1$ maps surjectively on $F$ (see Proposition 3.3.5 of \cite{I--P}). Lemma 3.3.4 of \cite{I--P} implies that the line $L$ corresponding to the  general point in $\Gamma_1$ satisfies $\cC_{L\vert F}\cong\cN_{L\vert F}^\vee\cong\cO_{\p1}^{\oplus2}$. Clearly the induced action of the automorphism group of $F$ on $\Gamma_1$ is transitive, hence there is an automorphism of $F$ fixing the class $h_2h_3$ and sending $E_{red}$ to $L$. We conclude that  $\cC_{E_{red}\vert F}\cong\cC_{L\vert F}\cong\cO_{\p1}^{\oplus2}$. 

On the other hand, the inclusion $E_{red}\subseteq E\subseteq F$ implies the existence of an epimorphism
$\cC_{E_{red}\vert F}\twoheadrightarrow \cC_{E_{red}\vert E}$. We obtain an epimorphism $\cO_{\p1}^{\oplus2}\twoheadrightarrow\cO_{\p1}(-1)$, clearly an absurd.

In particular both the cases Q and R are not admissible. In the former case the bundle $\cE$ exists, but it is decomposable, while in the latter $\cE$ does not exist at all.

We summarize the results of this section in the following result.
\begin{theorem}
\label{tIntermediate}
Let $\cE$ be an indecomposable initialized aCM bundle of rank $2$ on $F$ and let $c_1=\alpha_1h_1+\alpha_2h_2+\alpha_3h_3$ with $\alpha_1\le\alpha_2\le\alpha_3$. Assume that $c_1$ is neither $0$, nor $2h$, nor $h_1+2h_2+3h_3$. Then  $(c_1,c_2)$ is either $(h_3, h_2h_3)$, or $(h_1+2h_2+2h_3, 2h_2h_3+2h_1h_3+h_1h_2)$.

Conversely, for each such a pair, there exists an indecomposable, initialized, aCM bundle $\cE$ of rank $2$ on $F$ with Chern classes $c_1$ and $c_2$.

Moreover the zero--locus of a general section of $\cE$ is respectively a line or a, possibly reducible, quintic with arithmetic genus $0$. Each line on $F$ can be obtained in this way.
 \end{theorem}

\bigskip
\noindent
Gianfranco Casnati,\\
Dipartimento di Scienze Matematiche, Politecnico di Torino,\\
c.so Duca degli Abruzzi 24,\\
10129 Torino, Italy\\
e-mail: {\tt gianfranco.casnati@polito.it}

\bigskip
\noindent
Daniele Faenzi,\\
UFR des Sciences et des Techniques, Universit\'e de Bourgogne,\\
9 avenue Alain Savary -- BP 47870\\
21078 DIJON CEDEX, France\\
e-mail: {\tt daniele.faenzi@univ-pau.fr}

\bigskip
\noindent
Francesco Malaspina,\\
Dipartimento di Scienze Matematiche, Politecnico di Torino,\\
c.so Duca degli Abruzzi 24,\\
10129 Torino, Italy\\
e-mail: {\tt francesco.malaspina@polito.it}


\begin{thebibliography}{44}

\bibitem{A--O}
  V. Ancona, G. Ottaviani: \emph{Some applications of Beilinson's theorem to projective spaces and quadrics}. Forum Math. \textbf{3} \rm{(1991)}, 157--176.

\bibitem{Ar}
  E. Arrondo: \emph{A home--made Hartshorne--Serre correspondence}. Comm. Algebra  \textbf{ 20} \rm (2007),  423--443.

\bibitem{A--C}
  E. Arrondo, L. Costa: \emph{Vector bundles on Fano $3$--folds without intermediate cohomology}. Comm. Algebra  \textbf{ 28} \rm (2000),  3899--3911.

\bibitem{A--F}
  E. Arrondo, D. Faenzi: \emph{Vector bundles with no intermediate cohomology on Fano threefolds of type $V_{22}$}. Pacific J. Math. \textbf{225} \rm (2006), 201--220.

\bibitem{A--G}
  E. Arrondo and B. Gra\~na: \emph{Vector bundles on $G (1, 4)$ without intermediate cohomology}. J. Algebra \textbf{214} \rm(1999), 128--142.

\bibitem{A--M}  
  E. Arrondo, F. Malaspina:
  \emph{Cohomological Characterization of vector bundles on Grassmannians of lines}. J. of Algebra \textbf{323} \rm(2010), 1098--1106.

\bibitem{B--E}
  D. Bayer, D. Eisenbud: \emph{Ribbons and their canonical embeddings}. Trans. Amer. Math. Soc. \textbf{347} \rm (1995), 719--756.

\bibitem{B--M2}  
  E. Ballico, F. Malaspina: \emph{Regularity and cohomological splitting conditions for vector bundles on multiprojective spaces}.  J. of Algebra \textbf{345} \rm(2011), 137--149.

\bibitem{B--F1}
  C. Brambilla, D. Faenzi: \emph{Moduli space of rank $2$ ACM bundles on prime Fano threefolds},  Michigan Math. J. \textbf{60},  \rm(2011), 113--148.

\bibitem{C--H1}
  L. Casanellas, R. Hartshorne: \emph{ACM bundles on cubic surfaces}. J. Eur. Math. Soc. \textbf{13} \rm (2011), 709--731.

\bibitem{C--H2}
  M. Casanellas, R. Hartshorne, F. Geiss, F.O. Schreyer: \emph{Stable Ulrich bundles}. Int. J. of Math. \text{23} 1250083 \rm(2012).

\bibitem{Ca}
G. Casnati: \emph{On rank two bundles without intermediate cohomology}. Preprint.

\bibitem{C--F--M1}
G. Casnati, D.Faenzi, F. Malaspina: \emph{Moduli spaces of rank two aCM bundles on the Segre product of three projective lines}. arXiv:1404.1188 [math.AG].

\bibitem{C--F}
L. Chiantini, D. Faenzi: \emph{Rank 2 arithmetically Cohen-Macaulay
  bundles on a general quintic surface}. Math. Nachr. \textbf{282}
\rm (2009),  1691--1708.

\bibitem{C--M1}
  L. Chiantini, C. Madonna: \emph{A splitting criterion for rank 2 bundles on a general sextic threefold}. Internat. J. Math. \textbf{15} \rm (2004), 341--359.

\bibitem{C--M2}
  L. Chiantini, C. Madonna: \emph{ACM bundles on general hypersurfaces in ${\mathbb P}^5$ of low degree}. Collect. Math. \textbf{56} \rm(2005), 85--96.

\bibitem{C--G--N}
  N. Chiarli, S. Greco, U. Nagel: \emph{On the genus and Hartshorne--Rao module of projective curves}.  Math. Z.  \textbf{ 229} \rm (1998),  695--724.

\bibitem{CoCoA}
CoCoATeam: {\em CoCoA: a system for doing Computations in Commutative Algebra}.
Available at http://cocoa.dima.unige.it.

\bibitem{C--K--M}
  E. Coskun, R.S. Kulkarni, Y. Mustopa: \emph{The geometry of Ulrich bundles on del Pezzo surfaces}.  J. Alg. \textbf{375} \rm(2013), 280--301.


\bibitem{C--M--P} 
  L. Costa, R.M. Mir\'{o}--Roig, J. Pons--Llopis:
  \emph{ The representation type of Segre varieties}. Adv. Math. \textbf{230} \rm(2012), 1995--2013.

\bibitem{E--S} D. Eisenbud, F.-O. Schreyer, J. Weyman: \emph{Resultants and Chow forms via exterior syzygies}.
  J. Amer. Math. Soc. \textbf{16} \rm (2003),  537--579. 

\bibitem{Fa1}
  D. Faenzi: \emph{Bundles over the Fano threefold $V_5$}. Comm. Algebra \textbf{33} \rm(2005), 3061--3080.

\bibitem{Fa2}
  D. Faenzi: \emph{Even and odd instanton bundles on Fano threefolds
    of Picard number one}.  Manuscripta Math. \textbf{144} \rm (2014), 199–239.

\bibitem{Fu}
  W. Fulton: {\em Intersection theory}.  Ergebnisse der Mathematik und ihrer Grenzgebiete, Springer \rm (1984).
  
  \bibitem{G--H}
  Ph. Griffiths, J. Harris: {\em Principles of algebraic geometry}. Wiley Classics Library, Wiley \rm(1994).

\bibitem{Ha1}
  R. Hartshorne: \emph{Varieties of small codimension in projective space}. Bull. Amer. Math. Soc. \textbf{80} \rm (1974), 1017-1032.

\bibitem{Ha2}
  R. Hartshorne: {\em Algebraic geometry}. G.T.M. 52, Springer \rm (1977).

\bibitem{I--P}
  V.A. Iskovskikh, Yu.G. Prokhorov: \emph{Fano varieties}. Algebraic Geometry V (A.N. Parshin and I.R. Shafarevich eds.), Encyclopedia of Mathematical Sciences  47, Springer, \rm (1999).

\bibitem{Kno}
H. Kn\"orrer: \emph{
Cohen-Macaulay modules on hypersurface singularities. I}.
Invent. Math. \textbf{88} \rm (1987), 153-164. 

\bibitem{Ma1}
  C. Madonna: \emph{A splitting criterion for rank $2$ vector bundles on hypersurfaces in $\bP^4$}. Rend. Sem. Mat. Univ. Pol. Torino. \textbf{ 56} \rm (1998), 43--54.

\bibitem{Ma2}
  C. Madonna: \emph{Rank-two vector bundles on general quartic hypersurfaces in $\mathbb P^4$}. Rev. Mat. Complut. \textbf{13} (2000), 287--301.

\bibitem{Ma3}
  C. Madonna: \emph{ACM vector bundles on prime Fano threefolds and complete intersection Calabi Yau threefolds}. Rev. Roumaine Math. Pures Appl. \textbf{ 47} \rm (2002), 211--222.

\bibitem{Ma4}
  C. Madonna: \emph{Rank $4$ vector bundles on the quintic threefold}. Cent. Eur. J. Math. \textbf{3} \rm (2005), 404--411.

\bibitem{Ma}
  H. Matsumura: \emph{Commutative ring theory}. Cambridge U.P., \rm (1980).

\bibitem{MK--R--R1}
   N. Mohan Kumar, A. P. Rao, G. V. Ravindra: \emph{Arithmetically Cohen-Macaulay bundles on three dimensional hypersurfaces}. Int. Math. Res. Not. \rm(2007).

\bibitem{MK--R--R2}
 N. Mohan Kumar, A. P. Rao, G. V. Ravindra: \emph{Arithmetically Cohen-Macaulay bundles on hypersurfaces}. Comment. Math. Helv. \textbf{82} \rm(2007), 829–843.

\bibitem{Mu}
  D. Mumford: \emph{Lectures on curves on an algebraic surface. With a section by G. M. Bergman}. Annals of Mathematics Studies, No. 59 Princeton University Press, \rm(1966).

\bibitem{O--S--S}
  C. Okonek, M. Schneider, H. Spindler: \emph{ Vector bundles on complex projective spaces}. Progress in Mathematics 3, \rm(1980).

\bibitem{Ot}
  G. Ottaviani: \emph{Some extensions of Horrocks criterion to vector bundles on Grassmannians and quadrics}. Ann. Mat. Pura Appl. \textbf{ 155} \rm(1989), 317--341.

\bibitem{Ul}
B. Ulrich: \emph{Gorenstein rings and modules with high number of generators}. Math. Z. \textbf{188} \rm (1984) 23--32.

\bibitem{Vo}
  J.A. Vogelaar: {\em Constructing vector bundles from codimension-two subvarieties}, PhD Thesis.



\end{thebibliography}
\enddocument